\documentclass[amstex,12pt]{article}
\usepackage{amssymb,amsmath,amsthm,graphicx}
\newtheorem{theorem}{Theorem}[section]

\newtheorem{lemma}{Lemma}[section]
\newtheorem{definition}{Definition}[section]
\newtheorem{remark}{Remark}[section]
\newtheorem{example}{Example}[section]

\newcommand{\beq}{\begin{equation}}
\newcommand{\eeq}{\end{equation}}
\newcommand{\beqn}{\begin{eqnarray}}
\newcommand{\eeqn}{\end{eqnarray}}

\allowdisplaybreaks \baselineskip 20pt
\begin{document}

\title{The existence and global exponential stability of almost periodic solutions for neutral type CNNs  on time scales\thanks{This work is supported by the National Natural Sciences Foundation of People's Republic of China under Grant 11361072.} }
\author{Bing Li$^a$, Yongkun Li$^b$\thanks{%
The corresponding author.} and Xiaofang Meng$^b$ \\
$^a$School of Mathematics and Computer Science\\
 Yunnan Minzu University\\
 Kunming, Yunnan 650500\\
 People's Republic of China\\
$^b$Department of Mathematics, Yunnan University\\
Kunming, Yunnan 650091\\
People's Republic of China}

\date{}
\maketitle \allowdisplaybreaks
\begin{abstract}
In this paper, a class of neutral type competitive neural networks with mixed time-varying delays and leakage delays on time scales is proposed. Based on the exponential dichotomy of linear dynamic equations on time scales, Banach's fixed point theorem and the theory of calculus on time scales, some sufficient conditions that are independent of the backwards graininess function  of the time scale  are obtained for the existence and  global exponential stability of almost periodic solutions for this class of neural networks.
The obtained results are completely new and indicate that
both the continuous time and the discrete time cases of the networks share the same dynamical behavior. 
Finally, an examples is given to show the effectiveness of the obtained results.
\end{abstract}

\allowdisplaybreaks
\section{Introduction}

\setcounter{section}{1}
\setcounter{equation}{0}
\indent

The competitive neural networks (CNNs), which was first proposed by Cohen and Grossberg in \cite{bingf1}, model the dynamics of
cortical cognitive maps with unsupervised synaptic modifications.
In this model, there are two types of state
variables, the short-term memory variables (STM) describing
the fast neural activity and the long-term memory
(LTM) variables describing the slow unsupervised synaptic
modifications. Thus, there are two time scales in these
neural networks, one corresponds to the fast changes of
the neural network states, another corresponds to the slow
changes of the synapses by external stimuli. Since they are widely applied in the image processing, pattern recognition, signal
processing, optimization and control theory and so on \cite{bingf1,bingf2,bingf3,bingf4},
recently, many researchers paid attention to the dynamics analysis of CNNs [2-7]. For example, authors of \cite{b1,b2,b3,b4,bingda1,bingda4} studied the global stability for CNNs; authors of \cite{bingda2} investigate the  multistability  for CNNs;
authors of \cite{bingda6} investigated   multistability and multiperiodicity of  CNNs; authors of \cite{bingda3} studied the existence and global exponential stability of anti-periodic solutions for CNNs; authors of \cite{bingda5,bingda7} studied  the synchronization for CNNs  with mixed delays.
In reality, almost periodicity is much universal than periodicity. However, to the best of our knowledge, up to now, there are few papers published on the
existence of almost periodic solutions for CNNs, especially, for discrete time  CNNs.

In fact, it is natural and important that systems will contain some information about the derivative of the past state to further describe and model the dynamics for such complex neural reactions \cite{b15}, many authors investigated the dynamical behaviors of neutral type neural networks \cite{b15,bn1,bn2,bn3,bn4}. In reality, the mixed time-varying delay and leakage delay should be taken into account when modeling realistic neural networks \cite{bl1,bl2,bl3}.

In addition, it is well known that both continuous and discrete neural network systems are very important in implementation and applications. The theory of time scales, which was initialed by  Hilger \cite{d7} in his Ph.D. thesis, has recently received a lot of attention from many scholars. It not only unifies the continuous-time and discrete-time domains but also "between" them \cite{d7,bing1,bing2}. Therefore, it is necessary to study neural network systems on time scales.

On one hand, to the best of our knowledge, there is no published paper considering the global exponential stability of almost periodic solutions for CNNs with mixed time-varying delays and leakage delays on time scales. Therefore, it is a challenging and important problem in theories and applications.

On the other hand, in order to study the almost periodic dynamic equations on time scales, a concept of
almost periodic time scales was proposed in \cite{bing5}. Based on this concept, almost periodic
functions \cite{bing5}, pseudo almost periodic functions \cite{bbbli3},
almost automorphic functions \cite{bbbr1},   weighted pseudo
almost automorphic functions \cite{bbbli2},  almost periodic set-valued functions \cite{bbbset}, almost periodic functions in the sense
of Stepanov on time scales \cite{liwangpan} and so on were defined successively.  Also, some works have been done under the
concept of almost periodic time scales (see \cite{bbbr2,bbbr5,bbbr6,bbbr11,bbbr12,bbbr13,bbbr14,bbbr15}). Although  the concept of
almost periodic time scales in \cite{bing5} can unify the
continuous and discrete situations effectively, it is very restrictive because it requires the time scale with some global additivity. This
excludes many interesting time scales.
Therefore, it is a challenging and important problem in
theories and applications to study almost periodic problems on the time scale that  does not require such global additivity.

Motivated by above, in this paper, we propose the following competitive neural networks with mixed time-varying delays and leakage
delays on time scales:
\begin{eqnarray*}\label{aa1}
\left\{\begin{array}{lll}
STM: x_i^\nabla(t)&=&-\alpha_i(t)x_i(t-\eta_i(t))+\sum\limits_{j=1}^{n}D_{ij}(t)f_j(x_j(t))
+\sum\limits_{j=1}^{n}D_{ij}^\tau(t)f_j(x_j(t-\tau_{ij}(t)))\\
&&+\sum\limits_{j=1}^n\bar{D}_{ij}(t)\int_{t-\sigma_{ij}(t)}^{t}f_j(x_j(s))\nabla s+\sum\limits_{j=1}^n\tilde{D}_{ij}(t)\int_{t-\zeta_{ij}(t)}^{t}f_j(x_j^\nabla(s))\nabla s\\
&&+B_i(t)S_i(t)y_j+I_i(t),\\
LTM: m_{ij}^\nabla(t)&=&-c_i(t)m_{ij}(t-\varsigma_i(t))+y_jE_i(t)f_i(x_i(t))+J_i(t),
\end{array}\right.
\end{eqnarray*}
where $i=1,2,\ldots,n$, $\mathbb{T}$ is an almost periodic time scale, $x_{i}(t)$ is the neuron current activity level, $\alpha_{i}(t), c_{i}(t)$ are the time variable
of the neuron, $f_j(x_j(t))$ is the output of neurons, $m_{ij}(t)$
is the synaptic efficiency, $y_i$ is the constant external stimulus, $D_{ij}(t)$ and $D_{ij}^\tau(t)$, $\bar{D}_{ij}(t)$, $\tilde{D}_{ij}(t)$ represent the connection weight and the synaptic weight of delayed feedback between the $i$th neuron and the $j$th neuron respectively, $B_{i}(t)$ is the strength of the external stimulus, $E_{i}(t)$ denotes disposable scale, $I_{i}(t)$, $J_{i}(t)$ denote the external inputs on the $i$th neuron at time $t$, $\eta_{i}(t)$ and $\varsigma_{i}(t)$ are leakage delays and satisfy $t-\eta_{i}(t)\in \mathbb{T}$, $t-\varsigma_{i}(t)\in \mathbb{T}$ for $t\in \mathbb{T}$, $\tau_{ij}(t)$, $\sigma_{ij}(t)$ and $\zeta_{ij}(t)$ are transmission delays and satisfy $t-\tau_{ij}(t)\in \mathbb{T}$ , $t-\sigma_{ij}(t)\in \mathbb{T}$, $t-\zeta_{ij}(t)\in \mathbb{T}$ for $t\in \mathbb{T}$.

Setting $S_i=\sum\limits_{j=1}^{n}m_{ij}(t)y_i=\mathbf{y}^T \mathbf{m}_i(t)$, where  $\mathbf{y}=(y_1,y_2,\ldots,y_n)^T,\mathbf{ m}_i=(m_{i1},m_{i2},\ldots,m_{in})^T$
  and, without loss of generality,
the input stimulus $\mathbf{y}$ is assumed to be normalized with unit magnitude
 $|\mathbf{y}|^2=1$, summing up the LTM over $j$, then the above networks are simplified as the
networks:
\begin{eqnarray}\label{a1}
\left\{\begin{array}{lll}
STM: x_i^\nabla(t)&=&-\alpha_i(t)x_i(t-\eta_i(t))+\sum\limits_{j=1}^{n}D_{ij}(t)f_j(x_j(t))
+\sum\limits_{j=1}^{n}D_{ij}^\tau(t)f_j(x_j(t-\tau_{ij}(t)))\\
&&+\sum\limits_{j=1}^n\bar{D}_{ij}(t)\int_{t-\sigma_{ij}(t)}^{t}f_j(x_j(s))\nabla s+\sum\limits_{j=1}^n\tilde{D}_{ij}(t)\int_{t-\zeta_{ij}(t)}^{t}f_j(x_j^\nabla(s))\nabla s\\
&&+B_i(t)S_i(t)+I_i(t),\\
LTM: S_i^\nabla(t)&=&-c_i(t)S_i(t-\varsigma_i(t))+E_i(t)f_i(x_i(t))+J_i(t).
\end{array}\right.
\end{eqnarray}

For convenience,  for almost periodic functions on time scales, we introduce the following notations:
\[\alpha_{i}^{+}=\sup\limits_{t\in\mathbb{T}}|\alpha_{i}(t)|,\,\,\,\alpha_{i}^{-}=\inf\limits_{t\in\mathbb{T}}|\alpha_{i}(t)|
,\,\,\,c_{i}^{+}=\sup\limits_{t\in\mathbb{T}}|c_{i}(t)|,\,\,\,c_{i}^{-}=\inf\limits_{t\in\mathbb{T}}|c_{i}(t)|,\]
\[\eta_{i}^{+}=\sup\limits_{t\in\mathbb{T}}|\eta_{i}(t)|,\,\,\,\varsigma_{i}^{+}=\sup\limits_{t\in\mathbb{T}}|\varsigma_{i}(t)|,
\,\,\,D_{ij}^{+}=\sup\limits_{t\in\mathbb{T}}|D_{ij}(t)|,\,\,\,(D_{ij}^\tau)^{+}=\sup\limits_{t\in\mathbb{T}}|D_{ij}^\tau(t)|,\]
\[\bar{D}_{ij}^{+}=\sup\limits_{t\in\mathbb{T}}|\bar{D}_{ij}(t)|
,\,\,\,\tilde{D}_{ij}^{+}=\sup\limits_{t\in\mathbb{T}}|\tilde{D}_{ij}(t)|,\,\,\,B_{i}^{+}=\sup\limits_{t\in\mathbb{T}}|B_{i}(t)|,\,\,\,
E_{i}^{+}=\sup\limits_{t\in\mathbb{T}}|E_{i}(t)|,\]
\[\tau_{ij}^{+}=\sup\limits_{t\in\mathbb{T}}|\tau_{ij}(t)|,\,\,\,\sigma_{ij}^{+}=\sup\limits_{t\in\mathbb{T}}|\sigma_{ij}(t)|,\,\,\,
\zeta_{ij}^{+}=\sup\limits_{t\in\mathbb{T}}|\zeta_{ij}(t)|,\,\,\,i,j=1,2,\dots,n.\]

We denote $[a,b]_{\mathbb{T}}=\{t|t\in[a,b]\cap\mathbb{T}\}$. The  initial condition associated with system \eqref{a1} is of the form
\begin{eqnarray}\label{a2}
&&x_{i}(s)=\varphi_{i}(s),\,\,\,s\in[-\theta,0]_{\mathbb{T}},\\
&&S_{i}(s)=\phi_{i}(s),\,\,\,s\in[-\theta,0]_{\mathbb{T}},
\end{eqnarray}
where $i=1,2,\dots,n$, $\theta=\max\{\delta,\tau,\sigma,\zeta,\varsigma\}$, $\delta=\max\limits_{1\leq i\leq n}\{\delta_{i}^{+}\}$, $\tau=\max\limits_{1\leq i,j\leq n}\{\tau_{ij}^{+}\}$, $\sigma=\max\limits_{1\leq i,j\leq n}\{\sigma_{ij}^{+}\}$,
$\zeta=\max\limits_{1\leq i,j,\leq n}\{\zeta_{ij}^{+}\}$, $\varsigma=\max\limits_{1\leq i\leq n}\{\varsigma_{i}^{+}\}$, $i,j=1,2,\ldots,n$, $\varphi_{i}(\cdot)$, $\phi_{i}(\cdot)$ are real-valued bounded $\nabla$-differentiable
functions defined on $[-\theta,0]_{\mathbb{T}}$.

Our main purpose of this paper is to study the the existence and stability of almost periodic solutions of $\eqref{a1}$ on a new almost periodic time scale $\mathbb{T}$ which will be defined in the next section.

This paper is organized as follows. In Section 2, we introduce some definitions
and make some preparations for later sections
and we extend  the almost-periodic theory on time scales with the delta derivative in \cite{bing3} to that with the nabla derivative. In Section 3, by utilizing
Banach's fixed point theorem and the theory of calculus on time scales, we present
some sufficient conditions for the existence of almost periodic solutions of $\eqref{a1}$.
In Section 4, we prove that the almost periodic solution obtained in Section 3 is globally
exponentially stable. In Section 5, an examples are given to illustrate the effectiveness of the
theoretical results. Finally, we draw a conclusion in Section 6.

\section{Preliminaries}

\setcounter{section}{2}
\setcounter{equation}{0}
\indent

In this section, we shall first recall some fundamental definitions and lemmas in \cite{bing1,bing2}. Also, we extend the almost periodic theory on time scales with the delta derivative to that with the nabla derivative.

A time scale $\mathbb{T}$
is an arbitrary nonempty closed subset of the real set $\mathbb{R}$ with the topology and ordering inherited from $\mathbb{R}$.
The forward jump operator
$\sigma:\mathbb{T}\rightarrow\mathbb{T}$ is defined by $\sigma(t)=\inf\big\{s\in \mathbb{T},s>t\big\}$ for all $t\in \mathbb{T}$,
while the backward jump operator $\rho:\mathbb{T}\rightarrow\mathbb{T}$ is defined by
$\rho(t)=\sup\big\{s\in \mathbb{T},s<t\big\}$ for all $t\in\mathbb{T}$.

A point $t\in\mathbb{T}$ is called left-dense if $t>\inf\mathbb{T}$
and $\rho(t)=t$, left-scattered if $\rho(t)<t$, right-dense if
$t<\sup\mathbb{T}$ and $\sigma(t)=t$, and right-scattered if
$\sigma(t)>t$. If $\mathbb{T}$ has a left-scattered maximum $m$,
then $\mathbb{T}^k=\mathbb{T}\setminus\{m\}$; otherwise
$\mathbb{T}^k=\mathbb{T}$. If $\mathbb{T}$ has a right-scattered
minimum $m$, then $\mathbb{T}_k=\mathbb{T}\setminus\{m\}$; otherwise
$\mathbb{T}_k=\mathbb{T}$. Finally, the backwards graininess function
$\nu: \mathbb{T}_{k}\rightarrow [0,\infty)$ is defined by $\nu(t)=t-\rho(t)$.

A function $f:\mathbb{T}\rightarrow\mathbb{R}$ is ld-continuous provided
it is continuous at left-dense point in $\mathbb{T}$ and its right-side
limits exist at right-dense points in $\mathbb{T}$.

\begin{definition}$\cite{bing2}$
Let $f:\mathbb{T} \rightarrow \mathbb{R}$ be a function and
$t\in \mathbb{T}_{k}$. Then we define $f^{\nabla}(t)$ to be the
number (provided its exists) with the property that given any
$\varepsilon>0$, there is a neighborhood $U$ of
$t$ (i.e, $U=(t-\delta,t+\delta)\cap \mathbb{T}$ for some
$\delta>0$) such that
\[
|f(\rho(t))-f(s)-f^{\nabla}(t)(\rho(t)-s)|\leq\varepsilon|\rho(t)-s|
\]
for all $s\in U$, we call $f^{\nabla}(t)$ the nabla derivative of $f$ at $t$.
\end{definition}

Let $f:\mathbb{T} \rightarrow \mathbb{R}$ be ld-continuous. If $F^{\nabla}(t)=f(t)$, then we
define the nabla integral by $\int_a^{b}f(t)\nabla t=F(b)-F(a)$.

A function $p:\mathbb{T}\rightarrow\mathbb{R}$ is called $\nu$-regressive
if $1-\nu(t)p(t)\neq 0$ for all $t\in \mathbb{T}_k$. The set of all $\nu$-regressive and
left-dense continuous functions $p:\mathbb{T}\rightarrow\mathbb{R}$ will
be denoted by $\mathcal{R}_{\nu}=\mathcal{R}_{\nu}(\mathbb{T})=\mathcal{R}_{\nu}(\mathbb{T},\mathbb{R})$. We
define the set $\mathcal{R}_{\nu}^{+}=\mathcal{R}_{\nu}^{+}(\mathbb{T},\mathbb{R})=\{p\in \mathcal{R}_{\nu}:1-\nu(t)p(t)>0,\,\,\forall
t\in\mathbb{T}\}$.

If $p\in \mathcal {R}_{\nu}$, then we define the nabla exponential function by
\begin{equation*}
    \hat{e}_{p}(t,s)=\exp\bigg(\int_{s}^{t}\hat{\xi}_{\nu(\tau)}\big(p(\tau)\big)\nabla\tau\bigg),\,\,\mathrm{for}\,\,t,s\in\mathbb{T}
\end{equation*}
with the $\nu$-cylinder transformation
\[
\hat{\xi}_h(z)=\bigg\{\begin{array}{ll} -\frac{\mathrm{log}(1-hz)}{h} &{\rm
if}\,h\neq
0,\\
z &{\rm if}\,h=0.\\
\end{array}
\]

\begin{definition}$(\cite{bing2})$
Let $p,q\in\mathcal{R}_{\nu}$, then we define a circle plus addition by
$(p \oplus_{\nu} q)(t)= p(t)+q(t)-\nu(t) p(t)q(t)$, for all $t\in\mathbb{T}_{k}$. For $p\in\mathcal{R}_{\nu}$, define a circle minus
$p$ by $\ominus_{\nu} p=-\frac{p}{1-\nu p}$.
\end{definition}

\begin{lemma}$(\cite{bing2})$
Let $p,q\in\mathcal{R}_{\nu}$, and $s,t,r\in\mathbb{T}$. Then
\begin{itemize}
\item[$(i)$] $\hat{e}_{0}(t,s)\equiv 1$ and $\hat{e}_{p}(t,t)\equiv 1$;
\item[$(ii)$] $\hat{e}_{p}(\rho(t),s)=(1-\nu(t)p(t))\hat{e}_{p}(t,s)$;
\item[$(iii)$] $\hat{e}_{p}(t,s)=\frac{1}{\hat{e}_{p}(s,t)}=\hat{e}_{\ominus_{\nu} p}(s,t)$;
\item[$(iv)$] $\hat{e}_{p}(t,s)\hat{e}_{p}(s,r)=\hat{e}_{p}(t,r)$;
\item[$(v)$] $(\hat{e}_{p}(t,s))^{\nabla}=p(t)\hat{e}_{p}(t,s)$.
\end{itemize}
\end{lemma}

\begin{lemma}$(\cite{bing2})$ Let $f,g$ be nabla differentiable functions on $\mathbb{T}$, then
\begin{itemize}
\item[$(i)$] $(v_{1}f+v_{2}g)^{\nabla}=v_{1}f^{\nabla}+v_{2}g^{\nabla}$, for any constants $v_{1},v_{2}$;
\item[$(ii)$] $(fg)^{\nabla}(t)=f^{\nabla}(t)g(t)+f(\rho(t))g^{\nabla}(t)=f(t)g^{\nabla}(t)+f^{\nabla}(t)g(\rho(t))$;
\item[$(iii)$] If $f$ and $f^{\nabla}$ are continuous, then $(\int_{a}^{t}f(t,s)\nabla s)^{\nabla}=f(\rho(t),t)+\int_{a}^{t}f(t,s)\nabla s.$
\end{itemize}
\end{lemma}

\begin{lemma}$(\cite{bing2})$ Assume $p\in\mathcal{R}_{\nu}$ and $t_{0}\in\mathbb{T}$. If $1-\nu(t)p(t)>0$ for $t\in\mathbb{T}$, then $\hat{e}_{p}(t,t_{0})>0$ for all $t\in\mathbb{T}$.
\end{lemma}
\begin{definition}\cite{bing5} \label{def21} A time scale $\mathbb{T}$ is called an almost periodic time scale if
\begin{eqnarray*}
\Pi=\big\{\tau\in\mathbb{R}: t\pm\tau\in\mathbb{T}, \forall t\in{\mathbb{T}}\big\}\neq \{0\}.
\end{eqnarray*}
\end{definition}
\begin{definition} $(\cite{bing3,liwang1})$\label{def31}
A time scale $\mathbb{T}$ is called an almost periodic time scale if the set
$$\Lambda_0:=\big\{\tau\in \mathbb{R}:\mathbb{T}_{\pm \tau}\neq \emptyset\big\}\neq\{0\},$$
where $\mathbb{T}_{\pm\tau}=\mathbb{T}\cap\{\mathbb{T}\mp\tau\}=\mathbb{T}\cap \{t\mp\tau: t\in \mathbb{T}\}$, and there exists a set $\Lambda$ satisfying
\begin{itemize}
  \item [$(a)$]   $0\in\Lambda\subseteq \Lambda_0$,
  \item [$(b)$] $\Pi(\Lambda)\setminus\{0\}\neq \emptyset$,
  \item [$(c)$] $\widetilde{\mathbb{T}}:=\mathbb{T}(\Pi)=\bigcap\limits_{\tau\in\Pi}\mathbb{T}_\tau\neq \emptyset$,
\end{itemize}where $\Pi:=\Pi(\Lambda)=\{\tau\in \Lambda\subseteq \Lambda_0: \sigma\pm \tau\in \Lambda, \forall \sigma\in \Lambda\}$.
\end{definition}
Obviously, if $t\in \mathbb{T}_\tau$, then $t+\tau\in \mathbb{T}.$ If $t\in \widetilde{\mathbb{T}}$, then $t\pm\tau\in \mathbb{T}$ for $\tau\in \Pi$.
\begin{lemma}$(\cite{bing3})$\label{lem31}
If $\mathbb{T}$ is an almost periodic time scales under Definition \ref{def31}, then  $\widetilde{\mathbb{T}}$ is an almost periodic time scale under Definition \ref{def21}.
\end{lemma}
In the following, we restrict our discussion on an almost periodic time scale $\mathbb{T}$ that is defined by Definition \ref{def31}.
\begin{lemma}$(\cite{bing3})$\label{lem34}
If $t$ is a right-dense point of $\widetilde{\mathbb{T}}$, then $t$ is also a right-dense point of $\mathbb{T}$.
\end{lemma}
\begin{lemma}$(\cite{bing3})$\label{lem35}
If $t$ is a left-dense point of $\widetilde{\mathbb{T}}$, then $t$ is also a left-dense point of $\mathbb{T}$.
\end{lemma}

For each $f\in C(\mathbb{T},\mathbb{R})$, we define $\widetilde{f}:\widetilde{\mathbb{T}}\rightarrow\mathbb{R}$ by
$\widetilde{f}(t)=f(t)$ for $t\in\widetilde{\mathbb{T}}$. From Lemmas \ref{lem34} and   \ref{lem35}, we can get that
$\widetilde{f}\in C(\widetilde{\mathbb{T}},\mathbb{R})$. Therefore, $F$ defined by
\[
F(t):=\int^{t}_{t_0}\widetilde{f}(\tau)\widetilde{\nabla}\tau,\,\,
t_0,t\in\widetilde{\mathbb{T}}
\]
is an antiderivative of $f$ on $\widetilde{\mathbb{T}}$, where $\widetilde{\nabla}$ denotes the $\nabla$-derivative on $\widetilde{\mathbb{T}}$.

Let $BUC(\mathbb{T}\times D,\mathbb{R}^n)$ denote the collection of all bounded uniformly continuous functions from $\mathbb{T}\times S$ to $\mathbb{R}^n$, where $S\subset D$ is any compact set. We introduce the following definition of almost periodic functions on time scales as follows.
\begin{definition}$(\cite{bing3})$\label{def32}
Let $\mathbb{T}$ be an almost periodic time scale under sense of Definition \ref{def31}. A function  $f\in
BUC(\mathbb{T}\times D,\mathbb{E}^n)$  is  called an almost
periodic function in $t\in \mathbb{T}$ uniformly for $x\in D$ if the
$\varepsilon$-translation set of $f$
$$E\{\varepsilon,f,S\}=\{\tau\in{\Pi}:|f(t+\tau,x)-f(t,x)|<\varepsilon,\quad\forall  (t,x)\in   \widetilde{\mathbb{T}}\times S\}$$ is  relatively
dense  for all $\varepsilon>0$ and   for each
compact subset $S$ of $D$; that is, for any given $\varepsilon>0$
and each compact subset $S$ of $D$, there exists a constant
$l(\varepsilon,S)>0$ such that each interval of length
$l(\varepsilon,S)$ contains   a $\tau(\varepsilon,S)\in
E\{\varepsilon,f,S\}$ such that
\begin{equation*}
|f(t+\tau,x)-f(t,x)|<\varepsilon, \quad\forall (t,x)\in
\widetilde{\mathbb{T}}\times S.
\end{equation*}
This   $\tau$ is called the $\varepsilon$-translation number of $f$.
\end{definition}

\begin{definition}$(\cite{bing4})$
Let $A(t)$ be an $n\times n$ matrix-valued function on $\mathbb{T}$. Then the linear system
\begin{eqnarray}\label{e21}
x^{\nabla}(t)=A(t)x(t),\,\, t\in\mathbb{T}
\end{eqnarray}
is said to admit an exponential dichotomy on $\mathbb{T}$ if there
exist positive constant $K, \alpha$, projection $P$ and the
fundamental solution matrix $X(t)$ of \eqref{e21}, satisfying
\begin{eqnarray*}
\|X(t)PX^{-1}(s)\|_{0} \leq K\hat{e}_{\ominus_{\nu}
\alpha}(t,s),\,\,
s, t \in\mathbb{T}, t \geq s,\\
\|X(t)(I-P)X^{-1}(s)\|_{0} \leq K\hat{e}_{\ominus_{\nu}
\alpha}(s,t),\,\, s, t \in\mathbb{T}, t \leq s,
\end{eqnarray*}
where $\|\cdot\|_{0}$ is a matrix norm on $\mathbb{T}$ $($say, for
example, if $A=(a_{ij})_{n\times n}$, then we can take
$\|A\|_{0}=(\sum\limits_{i=1}^{n}\sum\limits_{j=1}^{n}|a_{ij}|^{2})^{\frac{1}{2}})$.
\end{definition}

Consider the following almost periodic system
\begin{eqnarray}\label{e22}
x^{\nabla}(t)=A(t)x(t)+f(t),\,\, t \in \mathbb{T},
\end{eqnarray}
where $A(t)$ is an almost periodic matrix function, $f(t)$ is an
almost periodic vector function.

\begin{lemma}$(\cite{bing4})$\label{blem24} If the linear system \eqref{e21} admits exponential dichotomy, then system
\eqref{e22} has a bounded solution $x(t)$ as follows:
\[
x(t)=\int_{-\infty}^{t}X(t)PX^{-1}(\rho(s))f(s)\nabla
s-\int_t^{+\infty}X(t)(I-P)X^{-1}(\rho(s))f(s)\nabla s,
\]
where $X(t)$ is the fundamental solution matrix of \eqref{e21}.
\end{lemma}

\begin{lemma}$(\cite{bing4})$ \label{zh}
Let $c_{i}(t)$ be an almost periodic function on $\mathbb{T}$, where
 $c_{i}\in \mathcal{R}_{\nu}^{+}$, and $\min\limits_{1 \leq i \leq n}\{\inf\limits_{t\in
\mathbb{T}}c_i(t)\}=\widetilde{m}
> 0$, then the linear system
\begin{eqnarray}\label{e23}
x^{\nabla}(t)=\mathrm{diag}(-c_{1}(t),-c_{2}(t),\dots,-c_{n}(t))x(t)
\end{eqnarray}
admits an exponential dichotomy on $\mathbb{T}$.
\end{lemma}

By Lemma 2.7 in \cite{bing4}, we have

\begin{lemma}\label{un}
Let $A(t)$ be an almost periodic matrix function and $f(t)$ be an almost periodic vector function. If \eqref{e21}
admits an exponential dichotomy, then \eqref{e22} has a unique almost periodic solution:
\[
x(t)=\int^t_{-\infty}\widetilde{X}(t)P\widetilde{X}^{-1}(\widetilde{\sigma}(s))\widetilde{f}(s)\widetilde{\nabla} s-\int_t^{+\infty}\widetilde{X}(t)(I-P)\widetilde{X}^{-1}(\widetilde{\sigma}(s))\widetilde{f}(s)\widetilde{\nabla} s,\,\, t\in \widetilde{\mathbb{T}},
\]
where $\widetilde{X}(t)$ is the restriction of the fundamental solution matrix of \eqref{e21} on $\widetilde{\mathbb{T}}$.
\end{lemma}

From Definition \ref{def32} and Lemmas \ref{blem24} and \ref{un}, one can easily get the following lemma.
\begin{lemma}\label{zh1}
If  linear system $\eqref{e21}$ admits an exponential dichotomy, then system $\eqref{e22}$ has an almost periodic solution
 $x(t)$ can be expressed as:
\begin{eqnarray*}
x(t)=\int_{-\infty}^{t}X(t)PX^{-1}(\sigma(s))f(s)\nabla s-\int^{+\infty}_{t}X(t)(I-P)X^{-1}(\sigma(s))f(s)\nabla s, \,\, t\in \mathbb{T},
\end{eqnarray*}
where $X(t)$ is the fundamental solution matrix of $\eqref{e21}$.
\end{lemma}

\section{The existence of almost periodic solution}

\setcounter{equation}{0}

\indent

In this section, we will state and prove the sufficient conditions for the existence of
almost periodic solutions of \eqref{a1}.

Let
\begin{eqnarray*}
\mathbb{B}&=&\big\{\psi=(\varphi_{1},\varphi_{2},\ldots,\varphi_{n},\phi_{1},\phi_{2},
\ldots,\phi_{n})^{T}: \\
&&\varphi_{i}, \phi_i \in C^1(\mathbb{T},\mathbb{R}),\,\,i=1,2,\ldots,n\big\}
\end{eqnarray*}
 with the norm
$\|\psi\|_{\mathbb{B}}=\sup\limits_{t\in\mathbb{T}}\|\psi(t)\|$ for $\|\psi(t)\|=\max\limits_{1\leq i\leq n}\{|\varphi_{i}(t)|,|\phi_{i}(t)|,|\varphi^{\nabla}_{i}(t)|,|\phi^{\nabla}_{i}(t)|\}$, then $\mathbb{B}$ is a Banach space.

Throughout the rest of this paper, we assume that the following conditions hold:
\begin{itemize}
\item[$(H_{1})$] $\alpha_{i}, c_{i}\in \mathcal{R}_{\nu}^{+},  D_{ij}, D_{ij}^{\tau}, \bar{D}_{ij}, \tilde{D}_{ij}, B_{i}, E_{i}, \eta_{i}, \varsigma_{i}, \tau_{ij}, \sigma_{ij}, \zeta_{ij}, I_{i}, J_{i}$ are continuous almost periodic functions for $i,j=1,2,\dots,n$;
\item[$(H_{2})$] The function $f_{j}\in C(\mathbb{R},\mathbb{R})$ and there exists positive constant  $L_{j}$ such that for all $x,y\in\mathbb{R}$
\[
 |f_{j}(x)-f_{j}(y)|\leq
L_{j}|x-y|,\,\,j=1,2,\dots,n.
\]
\end{itemize}

\begin{theorem}\label{thm31}
Let $(H_1)$-$(H_{2})$ hold.  Suppose that
\begin{itemize}
\item[$(H_{3})$] there exists a constant $r$ such that
\begin{eqnarray*}
&&\max\limits_{1\leq i\leq n}\bigg\{\frac{P_{i}}{\alpha_{i}^{-}},\bigg(1+\frac{\alpha_i^+}{\alpha_{i}^{-}}\bigg)P_{i},\frac{Q_{i}}{c_{i}^{-}},\bigg(1+\frac{c_i^+}{c_{i}^{-}}\bigg)Q_{i}
\bigg\}\leq r,\\
&& \max\limits_{1\leq i\leq n}\bigg\{\frac{\overline{P_{i}}}{\alpha_{i}^{-}},\bigg(1+\frac{\alpha_i^+}{\alpha_{i}^{-}}\bigg)\overline{P_{i}},\frac{\overline{Q_{i}}}{c_{i}^{-}},\bigg(1+\frac{c_i^+}{c_{i}^{-}}\bigg)\overline{Q_{i}}
\bigg\}:=\kappa< 1,
\end{eqnarray*}
where
\begin{eqnarray*}P_i&=&\alpha_{i}^{+}\eta_{i}^{+}r+\sum\limits_{j=1}^{n}D_{ij}^{+}(L_{j}r+|f_j(0)|)
+\sum\limits_{j=1}^{n}(D_{ij}^\tau)^{+}(L_{j}r+|f_j(0)|)\\
&&+\sum\limits_{j=1}^n\bar{D}_{ij}^{+}\sigma_{ij}^{+}(L_{j}r +|f_j(0)|)+\sum\limits_{j=1}^n\tilde{D}_{ij}^{+}\zeta_{ij}^{+}(L_{j}r+|f_j(0)|)
+B_i^{+}r+I_i^+,\\
Q_i& =&c_{i}^{+}\varsigma_{i}^{+}r
  +E_i^{+}(L_{i}r+|f_i(0)|)+J_i^+,\\
  \overline{P_i}&=&a_{i}^{+}\eta_{i}^{+}
  +\sum\limits_{j=1}^{n}D_{ij}^{+}L_{j} +\sum\limits_{j=1}^{n}(D_{ij}^\tau)^{+}L_{j} +\sum\limits_{j=1}^n\bar{D}_{ij}^{+}\sigma_{ij}^{+}L_{j}  +\sum\limits_{j=1}^n\tilde{D}_{ij}^{+}\zeta_{ij}^{+}L_{j} +B_i^{+}, \\
  \overline{Q_i}&=&c_{i}^{+}\varsigma_{i}^{+}
  +E_i^{+}L_{i},\,\,i=1,2,\ldots,n.
\end{eqnarray*}
\end{itemize}
Then system \eqref{a1} has  has a unique almost-periodic solution in the region
$\mathbb{E}=\{\psi\in\mathbb{B}: \|\psi\|_{\mathbb{B}}\leq r\}$.
\end{theorem}

\begin{proof}
Rewrite \eqref{a1} in the form
\begin{eqnarray*}
x_i^\nabla(t)&=&-\alpha_i(t)x_i(t)+\alpha_{i}(t)\int_{t-\eta_{i}(t)}^{t}x_{i}^{\nabla}(s)
\nabla s+\sum\limits_{j=1}^{n}D_{ij}(t)f_j(x_j(t))\\
&&+\sum\limits_{j=1}^{n}D_{ij}^\tau(t)f_j(x_j(t-\tau_{ij}(t)))
+\sum\limits_{j=1}^n\bar{D}_{ij}(t)\int_{t-\sigma_{ij}(t)}^{t}f_j(x_j(s))\nabla s \\ &&+\sum\limits_{j=1}^n\tilde{D}_{ij}(t)\int_{t-\zeta_{ij}(t)}^{t}f_j(x_j^\nabla(s))\nabla s
+B_i(t)S_i(t)+I_{i}(t),\,\,i=1,2,\dots,n,\\
S_i^\nabla(t)&=&-c_{i}(t)S_{i}(t)+c_i(t)\int_{t-\varsigma_{i}(t)}^{t}S_{i}^{\nabla}(s)
\nabla s+E_i(t)f_i(x_i(t))+J_{i}(t),\,\,i=1,2,\dots,n.
\end{eqnarray*}
For every $\psi=(\varphi_{1},\varphi_{2},\ldots,\varphi_{n},\phi_{1},\phi_{2},\ldots,\phi_{n})\in \mathbb{B}$, we consider the following system
\begin{eqnarray}\label{qb1}
  x_{i}^{\nabla}(t)&=&-\alpha_{i}(t)x_{i}(t)+F_{i}(t,\varphi_{i}),\,\,\, t\in \mathbb{T},\quad i=1,2,\ldots,n,\nonumber\\
  S_{i}^{\nabla}(t)&=&-c_{i}(t)S_{i}(t)+G_{i}(t,\phi_{i}),\,\,\,t\in \mathbb{T},\quad i=1,2,\ldots,n,\nonumber\\
\end{eqnarray}
where
\begin{eqnarray*}
F_{i}(t,\varphi_{i})&=&\alpha_{i}(t)\int_{t-\eta_{i}(t)}^{t}\varphi_{i}^{\nabla}(s)
\nabla s+\sum\limits_{j=1}^{n}D_{ij}(t)f_j(\varphi_j(t))\\
&&+\sum\limits_{j=1}^{n}D_{ij}^\tau(t)f_j(\varphi_j(t-\tau_{ij}(t)))
+\sum\limits_{j=1}^n\bar{D}_{ij}(t)\int_{t-\sigma_{ij}(t)}^{t}f_j(\varphi_j(s))\nabla s \\ &&+\sum\limits_{j=1}^n\tilde{D}_{ij}(t)\int_{t-\zeta_{ij}(t)}^{t}f_j(\varphi_j^\nabla(s))\nabla s
+B_i(t)\phi_i(t)+I_i(t),\quad i=1,2,\ldots,n,\\
G_{i}(t,\phi_{i})&=&c_i(t)\int_{t-\varsigma_{i}(t)}^{t}\phi_{i}^{\nabla}(s)
\nabla s+E_i(t)f_i(\varphi_i(t))+J_i(t),\quad i=1,2,\ldots,n.
\end{eqnarray*}
Since $\min\limits_{1\leq i\leq n}\big\{\inf\limits_{t\in
\mathbb{T}}\alpha_{i}(t), \inf\limits_{t\in\mathbb{T}}c_{i}(t)\big\}>0$, it follows from Lemma \ref{zh} that the linear system
\begin{eqnarray}\label{b2}
  x_{i}^{\nabla}(t)&=&-\alpha_{i}(t)x_{i}(t),\,\, i=1,2,\ldots,n,\nonumber\\
  S_{i}^{\nabla}(t)&=&-c_{i}(t)S_{i}(t),\,\,  i=1,2,\ldots,n
\end{eqnarray}
admits an exponential dichotomy on $\mathbb{T}$. Thus, by Lemma \ref{zh1}, we know that system \eqref{qb1} has exactly one almost periodic solution which can be expressed as follows:
\[
y_\psi(t)=( x_{\varphi_{1}}(t), x_{\varphi_{2}}(t),\ldots, x_{\varphi_{n}}(t),S_{\phi_{1}}(t),S_{\phi_{2}}(t),\ldots,S_{\phi_{n}}(t))^T,\]
where
\begin{eqnarray*}
  x_{\varphi_{i}}(t)=\int_{-\infty}^{t}\hat{e}_{-\alpha_{i}}(t,\rho(s))
  F_{i}(s,\varphi_{i})\nabla s,\,\, i=1,2,\ldots,n,\\
 S_{\phi_{i}}(t)=\int_{-\infty}^{t}\hat{e}_{-c_{i}}(t,\rho(s))
  G_{i}(t,\phi_{i})\nabla s,i=1,2,\ldots,n.
\end{eqnarray*}

Define the following operator $\Phi:\mathbb{E} \rightarrow \mathbb{E}$ by
\begin{eqnarray*}
\Phi(\psi(t))=y_\psi(t),\,\, \psi\in \mathbb{E}.
\end{eqnarray*}
We will show that $\Phi$ is a contraction.

First, we show that for any $\psi\in\mathbb{E}$, $\Phi\psi\in \mathbb{E}$. Note that for $i=1,2,\ldots,n$,
\begin{eqnarray*}
&&|F_{i}(s,\varphi_{i})|\\
&=&\bigg|\alpha_{i}(s)\int_{s-\eta_{i}(s)}^{s}\varphi_{i}^{\nabla}(u)
\nabla u+\sum\limits_{j=1}^{n}D_{ij}(s)f_j(\varphi_j(s))\\
&&+\sum\limits_{j=1}^{n}D_{ij}^\tau(s)f_j(\varphi_j(s-\tau_{ij}(s)))
+\sum\limits_{j=1}^n\bar{D}_{ij}(s)\int_{s-\sigma_{ij}(s)}^{s}f_j(\varphi_j(u))\nabla u \\ &&+\sum\limits_{j=1}^n\tilde{D}_{ij}(s)\int_{s-\zeta_{ij}(s)}^{s}f_j(\varphi_j^\nabla(u))\nabla u
+B_i(s)\phi_i(s)+I_i(s)\bigg|\\
  &\leq&\alpha_{i}^{+}\int_{s-\eta_{i}(s)}^{s}|\varphi_{i}^{\nabla}(u)| \nabla u+\sum\limits_{j=1}^{n}D_{ij}^{+}|f_j(\varphi_j(s))|\\
&&+\sum\limits_{j=1}^{n}(D_{ij}^\tau)^{+}|f_j(\varphi_j(s-\tau_{ij}(s)))|
+\sum\limits_{j=1}^n\bar{D}_{ij}^{+}\int_{s-\sigma_{ij}(s)}^{s}|f_j(\varphi_j(u))|\nabla u
\\ &&+\sum\limits_{j=1}^n\tilde{D}_{ij}^{+}\int_{s-\zeta_{ij}(s)}^{s}|f_j(\varphi_j^\nabla(u))|\nabla u
+B_i^{+}|\phi_i(s)|+I_i^+\\
&\leq&\alpha_{i}^{+}\eta_{i}^{+}r+\sum\limits_{j=1}^{n}D_{ij}^{+}(L_{j}|\varphi_j(s)|+|f_j(0)|)
+\sum\limits_{j=1}^{n}(D_{ij}^\tau)^{+}(L_{j}|\varphi_j(s-\tau_{ij}(s))|+|f_j(0)|)\\
&&+\sum\limits_{j=1}^n\bar{D}_{ij}^{+}\sigma_{ij}^{+}(L_{j}r+|f_j(0)|) +\sum\limits_{j=1}^n\tilde{D}_{ij}^{+}\zeta_{ij}^{+}(L_{j}r+|f_j(0)|)
+B_i^{+}|\phi_i(s)|+I_i^+\\
&\leq&\alpha_{i}^{+}\eta_{i}^{+}r+\sum\limits_{j=1}^{n}D_{ij}^{+}(L_{j}r+|f_j(0)|)
+\sum\limits_{j=1}^{n}(D_{ij}^\tau)^{+}(L_{j}r+|f_j(0)|)\\
&&+\sum\limits_{j=1}^n\bar{D}_{ij}^{+}\sigma_{ij}^{+}(L_{j}r +|f_j(0)|)+\sum\limits_{j=1}^n\tilde{D}_{ij}^{+}\zeta_{ij}^{+}(L_{j}r+|f_j(0)|)
+B_i^{+}r+I_i^+=P_i
\end{eqnarray*}
and
\begin{eqnarray*}
|G_{i}(s,\phi_{i})|&=&\bigg|c_i(s)\int_{s-\varsigma_{i}(s)}^{s}\phi_{i}^{\nabla}(u)
\nabla s+E_i(s)f_i(\varphi_i(s))+J_i(s)\bigg|\\
  &\leq&c_{i}^{+}\int_{s-\varsigma_{i}(s)}^{s}|\phi_{i}^{\nabla}(u)| \nabla u
  +E_i^{+}|f_i(\varphi_i(s))|+J_i^+\\
  &\leq&c_{i}^{+}\varsigma_{i}^{+}|\phi_{i}^{\nabla}(s)|
  +E_i^{+}(L_{i}|\varphi_i(s)|+|f_i(0)|)+J_i^+\\
  &\leq&c_{i}^{+}\varsigma_{i}^{+}r
  +E_i^{+}(L_{i}r+|f_i(0)|)+J_i^+=Q_i.
\end{eqnarray*}
Therefore, we can get for $i=1,2,\ldots,n$,
\begin{eqnarray*}
&&\sup\limits_{t\in\mathbb{T}}|x_{\varphi_{i}}(t)|\\
&=&\sup\limits_{t\in\mathbb{T}}\bigg|\int_{-\infty}^{t}\hat{e}_{-\alpha_{i}}(t,\rho(s))
 F_{i}(s,\varphi_{i})\nabla s\bigg|\\
  &\leq&\sup\limits_{t\in\mathbb{T}}\int_{-\infty}^{t}\hat{e}_{-\alpha_{i}^{-}}(t,\rho(s))
  \big|F_{i}(s,\varphi_{i})\big|\nabla s\\
  &\leq&\frac{1}{\alpha_{i}^{-}}\bigg[\bigg(\alpha_{i}^{+}\eta_{i}^{+} +\sum\limits_{j=1}^{n}D_{ij}^{+}(L_{j}+|f_j(0)|)
+\sum\limits_{j=1}^{n}(D_{ij}^\tau)^{+}(L_{j}+|f_j(0)|)\\
&&+\sum\limits_{j=1}^n\bar{D}_{ij}^{+}\sigma_{ij}^{+}(L_{j} +|f_j(0)|) +\sum\limits_{j=1}^n\tilde{D}_{ij}^{+}\zeta_{ij}^{+}(L_{j}+|f_j(0)|)
+B_i^{+}\bigg)r+I_{i}^{+}\bigg]=\frac{P_i}{\alpha_i^-}
\end{eqnarray*}
and
\begin{eqnarray*}
  \sup\limits_{t\in\mathbb{T}}|S_{\phi_{i}}(t)|&=&\sup\limits_{t\in\mathbb{T}}\bigg|\int_{-\infty}^{t}\hat{e}_{-c_{i}}(t,\rho(s))
  G_{i}(t,\phi_{i})\nabla s\bigg|\\
  &\leq&\sup\limits_{t\in\mathbb{T}}\int_{-\infty}^{t}\hat{e}_{-c_{i}^{-}}(t,\rho(s))
  \big|G_{i}(s,\phi_{i})\big|\nabla s\\
  &\leq&\frac{1}{c_{i}^{-}}\big(c_{i}^{+}\varsigma_{i}^{+}r
  +E_i^{+}(L_{i}r+|f_i(0)|)+J_{i}^{+}\big)=\frac{Q_i}{c_i^-}.
\end{eqnarray*}

On the other hand, for $i=1,2,\ldots,n$, we have
\begin{eqnarray*}
  &&\sup\limits_{t\in\mathbb{T}}|x_{\varphi_{i}}^{\nabla}(t)|\\
  &=&\sup\limits_{t\in\mathbb{T}}\bigg|\bigg(\int_{-\infty}^{t}\hat{e}_{-\alpha_{i}}(t,\rho(s))
  F_{i}(s,\varphi_{i})\nabla s\bigg)^{\nabla}\bigg|\\
  &=&\sup\limits_{t\in\mathbb{T}}\bigg|
  F_{i}(t,\varphi_{i})-\alpha_{i}(t)\int_{-\infty}^{t}\hat{e}_{-\alpha_{i}}(t,\rho(s))
  F_{i}(s,\varphi_{i})\nabla s \bigg|\\
 &\leq& \alpha_{i}^{+}\eta_{i}^{+}r +\sum\limits_{j=1}^{n}D_{ij}^{+}(L_{j}r+|f_j(0)|)
+\sum\limits_{j=1}^{n}(D_{ij}^\tau)^{+}(L_{j}r+|f_j(0)|)\\
&&+\sum\limits_{j=1}^n\bar{D}_{ij}^{+}\sigma_{ij}^{+}(L_{j}r+|f_j(0)|) +\sum\limits_{j=1}^n\tilde{D}_{ij}^{+}\zeta_{ij}^{+}(L_{j}r+|f_j(0)|)
+B_i^{+}r\\
&& +I_{i}^{+}+\frac{\alpha_{i}^{+}}{\alpha_i^-}\bigg(
  \alpha_{i}^{+}\eta_{i}^{+}+\sum\limits_{j=1}^{n}D_{ij}^{+}(L_{j}r+|f_j(0)|)+\sum\limits_{j=1}^{n}(D_{ij}^\tau)^{+}(L_{j}r+|f_j(0)|)\\
&&
+\sum\limits_{j=1}^n\bar{D}_{ij}^{+}\sigma_{ij}^{+}(L_{j}r+|f_j(0)|) +\sum\limits_{j=1}^n\tilde{D}_{ij}^{+}\zeta_{ij}^{+}(L_{j}r+|f_j(0)|)
+B_i^{+}r\bigg) +\frac{\alpha_i^+I_{i}^{+}}{\alpha_{i}^{-}}\\
&=&\bigg(1+\frac{\alpha_{i}^{+}}{\alpha_i^-}\bigg)\bigg(
  \alpha_{i}^{+}\eta_{i}^{+}r+\sum\limits_{j=1}^{n}D_{ij}^{+}(L_{j}r+|f_j(0)|)+\sum\limits_{j=1}^{n}(D_{ij}^\tau)^{+}(L_{j}r+|f_j(0)|)\\
&&
+\sum\limits_{j=1}^n\bar{D}_{ij}^{+}\sigma_{ij}^{+}(L_{j}r+|f_j(0)|) +\sum\limits_{j=1}^n\tilde{D}_{ij}^{+}\zeta_{ij}^{+}(L_{j}r+|f_j(0)|)
+B_i^{+}r +{I_{i}^{+}}\bigg)\\
&=&\bigg(1+\frac{\alpha_{i}^{+}}{\alpha_i^-}\bigg)P_i
\end{eqnarray*}
and
\begin{eqnarray*}
  &&\sup\limits_{t\in\mathbb{T}}|S_{\phi_{i}}^{\nabla}(t)|\\
  &=&\bigg(\int_{-\infty}^{t}\hat{e}_{-c_{i}}(t,\rho(s))
  G_{i}(t,\phi_{i})\nabla s\bigg)^{\nabla}\\
  &=&\sup\limits_{t\in\mathbb{T}}\bigg|
  G_{i}(t,\phi_{i})-c_{i}(t)\int_{-\infty}^{t}\hat{e}_{-c_{i}}(t,\rho(s))
  G_{i}(t,\phi_{i})\nabla s\bigg|\\
  &\leq&c_{i}^{+}\varsigma_{i}^{+}r
  +E_i^{+}(L_{i}r+|f_i(0)|)+J_{i}^{+}+\frac{c_{i}^{+}}{c_i^-}
  \big(c_{i}^{+}\varsigma_{i}^{+}r
  +E_i^{+}(L_{i}r+|f_i(0)|)\big)+\frac{c_i^+J_{i}^{+}}{c_{i}^{-}}\\
  &=&  \bigg(1+\frac{c_{i}^{+}}{c_i^-}\bigg)
  \big( c_{i}^{+}\varsigma_{i}^{+}r
  +E_i^{+}(L_{i}r+|f_i(0)|) +J_{i}^{+}\big)\\
  &=&  \bigg(1+\frac{c_{i}^{+}}{c_i^-}\bigg)
  Q_i.
\end{eqnarray*}
In view of $(H_{3})$, we have
\begin{eqnarray*}
\|\Phi(\psi)\|_{\mathbb{B}}\leq r,
\end{eqnarray*}
which implies that $\Phi\psi\in \mathbb{E}$, so the mapping $\Phi$ is
a self-mapping from $\mathbb{E}$ to $\mathbb{E}$.

Next, we shall prove that $\Phi$ is a contraction mapping.
For any $\psi=(\varphi_{1},\varphi_{2},\ldots,\varphi_{n},\phi_{1},\phi_{2},\\
\ldots,\phi_{n})^T,\omega=(u_{1},u_{2},\ldots,u_{n},v_{1},v_{2},\ldots,v_{n})^T\in \mathbb{E}$, we have for $i=1,2,\ldots,n$,
\begin{eqnarray*}
\sup\limits_{t\in\mathbb{T}}\big|x_{\varphi_{i}}(t)-x_{u_{i}}(t)\big|
  &\leq&\frac{1}{\alpha_{i}^{-}}\bigg(a_{i}^{+}\eta_{i}^{+}|\varphi_{i}^{\nabla}(s)-u_{i}^{\nabla}(s)|
  +\sum\limits_{j=1}^{n}D_{ij}^{+}L_{j}|\varphi_j(s)-u_j(s)|\\
&&+\sum\limits_{j=1}^{n}(D_{ij}^\tau)^{+}L_{j}|\varphi_j(s-\tau_{ij}(s))-u_j(s-\tau_{ij}(s))|\\
&&+\sum\limits_{j=1}^n\bar{D}_{ij}^{+}\sigma_{ij}^{+}L_{j}|\varphi_j(s)-u_j(s)| +\sum\limits_{j=1}^n\tilde{D}_{ij}^{+}\zeta_{ij}^{+}L_{j}|\varphi_j^\nabla(s)-u_j^\nabla(s)|\\
&&+B_i^{+}|\phi_i(s)-u_i(s)|\bigg)\\
&\leq&\frac{1}{\alpha_{i}^{-}}\bigg(a_{i}^{+}\eta_{i}^{+}
  +\sum\limits_{j=1}^{n}D_{ij}^{+}L_{j} +\sum\limits_{j=1}^{n}(D_{ij}^\tau)^{+}L_{j} \\
&&+\sum\limits_{j=1}^n\bar{D}_{ij}^{+}\sigma_{ij}^{+}L_{j}  +\sum\limits_{j=1}^n\tilde{D}_{ij}^{+}\zeta_{ij}^{+}L_{j} +B_i^{+} \bigg)||\psi-\omega||_\mathbb{B}\\
&=&\frac{\overline{P_i}}{\alpha_{i}^{-}}||\psi-\omega||_\mathbb{B}
\end{eqnarray*}
and
\begin{eqnarray*}
\sup\limits_{t\in\mathbb{T}}\big|(x_{\varphi_{i}}(t)-x_{u_{i}}(t))^{\nabla}\big|
  &\leq&\sup\limits_{s\in\mathbb{T}}\bigg[a_{i}^{+}\eta_{i}^{+}|\varphi_{i}^{\nabla}(s)-u_{i}^{\nabla}(s)|
  +\sum\limits_{j=1}^{n}D_{ij}^{+}L_{j}|\varphi_j(s)-u_j(s)|\\
&&+\sum\limits_{j=1}^{n}(D_{ij}^\tau)^{+}L_{j}|\varphi_j(s-\tau_{ij}(s))-u_j(s-\tau_{ij}(s))|\\
&&+\sum\limits_{j=1}^n\bar{D}_{ij}^{+}\sigma_{ij}^{+}L_{j}|\varphi_j(s)-u_j(s)| +\sum\limits_{j=1}^n\tilde{D}_{ij}^{+}\zeta_{ij}^{+}L_{j}|\varphi_j^\nabla(s)-u_j^\nabla(s)|\\
&&+B_i^{+}|\phi_i(s)-u_i(s)|\\
 &&+\frac{\alpha_{i}^{+}}{\alpha_{i}^{-}}
  \bigg(a_{i}^{+}\eta_{i}^{+}|\varphi_{i}^{\nabla}(s)-u_{i}^{\nabla}(s)|
  +\sum\limits_{j=1}^{n}D_{ij}^{+}L_{j}|\varphi_j(s)-u_j(s)|\\
&&+\sum\limits_{j=1}^{n}(D_{ij}^\tau)^{+}L_{j}|\varphi_j(s-\tau_{ij}(s))-u_j(s-\tau_{ij}(s))|\\
&&+\sum\limits_{j=1}^n\bar{D}_{ij}^{+}\sigma_{ij}^{+}L_{j}|\varphi_j(s)-u_j(s)| +\sum\limits_{j=1}^n\tilde{D}_{ij}^{+}\zeta_{ij}^{+}L_{j}|\varphi_j^\nabla(s)-u_j^\nabla(s)|\\
&&+B_i^{+}|\phi_i(s)-u_i(s)|\bigg)\bigg]\\
 &\leq&\bigg(a_{i}^{+}\eta_{i}^{+}
  +\sum\limits_{j=1}^{n}D_{ij}^{+}L_{j} +\sum\limits_{j=1}^{n}(D_{ij}^\tau)^{+}L_{j} \\
&&+\sum\limits_{j=1}^n\bar{D}_{ij}^{+}\sigma_{ij}^{+}L_{j} +\sum\limits_{j=1}^n\tilde{D}_{ij}^{+}\zeta_{ij}^{+}L_{j} +B_i^{+}\\
 &&+\frac{\alpha_{i}^{+}}{\alpha_{i}^{-}}
  \bigg(a_{i}^{+}\eta_{i}^{+}
  +\sum\limits_{j=1}^{n}D_{ij}^{+}L_{j} +\sum\limits_{j=1}^{n}(D_{ij}^\tau)^{+}L_{j} \\
&&+\sum\limits_{j=1}^n\bar{D}_{ij}^{+}\sigma_{ij}^{+}L_{j}  +\sum\limits_{j=1}^n\tilde{D}_{ij}^{+}\zeta_{ij}^{+}L_{j} +B_i^{+} \bigg)\bigg)||\psi-\omega||_\mathbb{B}\\
&=&\bigg(1+\frac{\alpha_{i}^{+}}{\alpha_{i}^{-}}\bigg)
  \bigg(a_{i}^{+}\eta_{i}^{+}
  +\sum\limits_{j=1}^{n}D_{ij}^{+}L_{j} +\sum\limits_{j=1}^{n}(D_{ij}^\tau)^{+}L_{j} \\
&&+\sum\limits_{j=1}^n\bar{D}_{ij}^{+}\sigma_{ij}^{+}L_{j}  +\sum\limits_{j=1}^n\tilde{D}_{ij}^{+}\zeta_{ij}^{+}L_{j} +B_i^{+} \bigg)||\psi-\omega||_\mathbb{B}\\
&=&\bigg(1+\frac{\alpha_{i}^{+}}{\alpha_{i}^{-}}\bigg)\overline{P_i}
  ||\psi-\omega||_\mathbb{B}.
\end{eqnarray*}
In a similar way, we have for $i=1,2,\ldots,n$,
\begin{eqnarray*}
\sup\limits_{t\in\mathbb{T}}\big|S_{\phi_{i}}(t)-S_{v_{i}}(t)\big|
&\leq&\frac{1}{c_{i}^{-}}\big(c_{i}^{+}\varsigma_{i}^{+}
  +E_i^{+}L_{i} \big)||\psi-\omega||_\mathbb{B}\\
  &=&\frac{1}{c_{i}^{-}}\overline{Q_i}||\psi-\omega||_\mathbb{B}
  \end{eqnarray*}
  and
  \begin{eqnarray*}
\sup\limits_{t\in\mathbb{T}}\big|(S_{\phi_{i}}(t)-S_{v_{i}}(t))^{\nabla}\big|
&\leq&\bigg(1+\frac{c_{i}^{+}}{c_{i}^{-}}\bigg)\big(c_{i}^{+}\varsigma_{i}^{+}
  +E_i^{+}L_{i} \big)||\psi-\omega||_\mathbb{B}\\
  &=&\bigg(1+\frac{c_{i}^{+}}{c_{i}^{-}}\bigg)\overline{Q_i}||\psi-\omega||_\mathbb{B}.
\end{eqnarray*}
By $(H_{3})$, we have
\begin{eqnarray*}
\|\Phi(\psi)-\Phi(\omega)\|_{\mathbb{B}}\leq \kappa\|\psi-\omega\|_{\mathbb{B}}.
\end{eqnarray*}
Hence, we obtain that $\Phi$ is a contraction mapping. Then, system \eqref{a1} has a unique almost periodic solution in the region
$\mathbb{E}=\{\psi\in\mathbb{B}: \|\psi\|_{\mathbb{B}}\leq r\}$. This completes the proof of Theorem \ref{thm31}.
\end{proof}

\section{Global exponential stability of almost periodic solution}

\setcounter{equation}{0}

\indent

In this section, we will study the exponential stability of almost periodic solutions of \eqref{a1}.

\begin{definition}
The almost periodic solution $Z^{\ast}(t)=(x_{1}^{\ast}(t),x_{2}^{\ast}(t),\ldots,x_{n}^{\ast}(t),S_{1}^{\ast}(t),S_{2}^{\ast}(t),\ldots,\\S_{n}^{\ast}(t))^{T}$ of
system \eqref{a1} with initial value $\psi^{\ast}(s)=(\varphi_{1}^{\ast}(s),\varphi_{2}^{\ast}(s),\ldots,\varphi_{n}^{\ast}(s),\phi_{1}^{\ast}(s), \phi_{2}^{\ast}(s),
\ldots,\\
\phi_{n}^{\ast}(s))^{T}$ is said to be globally exponentially stable if there exist   positive constants $\lambda$ with $\ominus\lambda\in \mathcal{R}^+$ and $M>1$ such that every solution
$Z(t)=(x_{1}(t),x_{2}(t),\ldots,x_{n}(t),S_{1}(t),S_{2}(t),\ldots,S_{n}(t))^{T}$
of system \eqref{a1} with initial value $\psi(s)=(\varphi_{1}(s),\varphi_{2}(s),\ldots,\varphi_{n}(s),\phi_{1}(s), \phi_{2}(s),
\ldots,\phi_{n}(s))^{T}$
satisfies
\[\|Z(t)-Z^{\ast}(t)\|\leq Me_{\ominus\lambda}(t,t_0)\|\psi-\psi^{\ast}\|_0,\,\,\,\,\forall t\in(0,+\infty)_{\mathbb{T}},\,\,t\geq t_{0},\]
where $\|Z(t)-Z^{\ast}(t)\|=\max\limits_{1\leq i\leq n}\big\{|x_{i}(t)-x_{i}^{\ast}(t)|,S_{i}(t)-S_{i}^{\ast}(t)|,
|(x_{i}(t)-x_{i}^{\ast}(t))^{\nabla}|,|(S_{i}(t)-S_{i}^{\ast}(t)|)^{\nabla}\big\}$, $\|\psi-\psi^{\ast}\|_0=\max\limits_{1\leq i\leq n}\bigg\{\sup\limits_{s\in[-\theta,0]_{\mathbb{T}}}|\varphi_{i}(s)-\varphi_{i}^{\ast}(s)|,
\sup\limits_{s\in[-\theta,0]_{\mathbb{T}}}|\phi_{i}(s)-\phi_{i}^{\ast}(s)|,\sup\limits_{s\in[-\theta,0]_{\mathbb{T}}}|(\varphi_{i}(s)-\varphi_{i}^{\ast}(s))^\nabla|,\\
\sup\limits_{s\in[-\theta,0]_{\mathbb{T}}}|(\phi_{i}(s)-\phi_{i}^{\ast}(s))^\nabla|\bigg\}$ .
\end{definition}

\begin{theorem}\label{thm41}
Assume that $\sup\limits_{t\in\mathbb{T}}\nu(t)<+\infty$ and $(H_1)$-$(H_3)$ hold, then system \eqref{a1} has a unique almost periodic solution which is globally exponentially stable.
\end{theorem}
\begin{proof}
From Theorem \ref{thm31},  we see that system \eqref{a1} has an almost periodic solution $Z^{\ast}(t)=(x_{1}^{\ast}(t),x_{2}^{\ast}(t),\ldots,x_{n}^{\ast}(t),S_{1}^{\ast}(t),S_{2}^{\ast}(t),\ldots,S_{n}^{\ast}(t))^{T}$ with the initial value $\psi^{\ast}(s)=(\varphi_{1}^{\ast}(s),\varphi_{2}^{\ast}(s),\ldots,\\
\varphi_{n}^{\ast}(s),\phi_{1}^{\ast}(s), \phi_{2}^{\ast}(s),
\ldots,\phi_{n}^{\ast}(s))^{T}$. Suppose that $Z(t)=(x_{1}(t),x_{2}(t), \ldots,x_{n}(t),
S_{1}(t),S_{2}(t),\ldots,\\
S_{n}(t))^{T}$ is an arbitrary solution of \eqref{a1} with the initial value
$\psi(s)=(\varphi_{1}(s),
\varphi_{2}(s),\ldots,\varphi_{n}(s),\phi_{1}(s),\\
 \phi_{2}(s),
\ldots,\phi_{n}(s))^{T}$.

Then it follows from system \eqref{a1} that
\begin{eqnarray}\label{c1}
\left\{\begin{array}{lll}
u_i^\nabla(t)&=&-\alpha_i(t)u_i(t)+\alpha_{i}(t)\int_{t-\eta_{i}(t)}^{t}u_{i}^{\nabla}(s)
\nabla s+\sum\limits_{j=1}^{n}D_{ij}(t)p_j(u_j(t))\\
&&+\sum\limits_{j=1}^{n}D_{ij}^\tau(t)p_j(u_j(t-\tau_{ij}(t)))
+\sum\limits_{j=1}^n\bar{D}_{ij}(t)\int_{t-\sigma_{ij}(t)}^{t}p_j(u_j(s))\nabla s \\ &&+\sum\limits_{j=1}^n\tilde{D}_{ij}(t)\int_{t-\zeta_{ij}(t)}^{t}h_j(u_j^\nabla(s))\nabla s
+B_i(t)v_i(t),\,\,i=1,2,\dots,n,\\
v_i^\nabla(t)&=&-c_{i}(t)v_{i}(t)+c_i(t)\int_{t-\varsigma_{i}(t)}^{t}u_{i}^{\nabla}(s)
\nabla s+E_i(t)p_i(u_i(t)),\,\,i=1,2,\dots,n,
\end{array}\right.
\end{eqnarray}
where $u_{i}(t)=x_{i}(t)-x^{\ast}_{i}(t)$, $v_{i}(t)=S_{i}(t)-S^{\ast}_{i}(t)$, $p_j(u_j(t)))=f_j(x_j(t))-f_j(x_j^*(t)), h_j(x_j^\nabla(s))=f_j(x_j^\nabla(s))-f_j({x_j^*}^\nabla(s))$ and $i,j=1,2,\ldots,n$.

The initial condition of \eqref{c1} is
\begin{eqnarray*}
&&u_{i}(s)=\varphi_{i}(s)-\varphi_{i}^{\ast}(s),\,\,\,s\in[-\theta,0]_{\mathbb{T}},\\
&&v_{i}(s)=\phi_{i}(s)-\phi_{i}^{\ast}(s),\,\,\,s\in[-\theta,0]_{\mathbb{T}},
\end{eqnarray*}
where $i=1,2,\dots,n$.

Multiplying   both sides of \eqref{c1} by $\hat{e}_{-\alpha_{i}}(t_0,\rho(s))$ and $\hat{e}_{-c_{i}}(t_{0},\rho(s))$, respectively, and integrating over $[t_{0},t]_{\mathbb{T}}$, where $t_{0}\in[-\theta,0]_{\mathbb{T}}$, we have
\begin{eqnarray}\label{c2}
\left\{\begin{array}{lll}
 u_{i}(t)&=& u_{i}(t_{0})\hat{e}_{-\alpha_{i}}(t,t_{0})
 +\int_{t_{0}}^{t}\hat{e}_{-\alpha_{i}}(t,\rho(s))\bigg(\alpha_{i}(s)\int_{s-\eta_{i}(s)}^{s}u_{i}^{\nabla}(u)
\nabla u+\sum\limits_{j=1}^{n}D_{ij}(s)p_j(u_j(s))\\
&&+\sum\limits_{j=1}^{n}D_{ij}^\tau(s)p_j(u_j(s-\tau_{ij}(s)))
+\sum\limits_{j=1}^n\bar{D}_{ij}(s)\int_{s-\sigma_{ij}(s)}^{s}p_j(u_j(u))\nabla u \\ &&+\sum\limits_{j=1}^n\tilde{D}_{ij}(s)\int_{s-\zeta_{ij}(s)}^{s}h_j(u_j^\nabla(u))\nabla u
+B_i(s)v_i(s)\bigg)\nabla s,\\
v_{i}(t)&=&v_{i}(t_{0})\hat{e}_{-c_{i}}(t,t_{0})
+\int_{t_{0}}^{t}\hat{e}_{-c_{i}}(t,\rho(s))\Big(c_{i}(s)\int_{s-\varsigma_{i}(s)}^{s}v_{i}^{\nabla}(u)
\nabla u+E_i(s)p_i(u_i(s))\Big)\nabla s,
\end{array}\right.
\end{eqnarray}
where $i=1,2,\dots,n$.

For $i=1,2,\dots,n$, let $H_{i}, \overline{H_{i}}$, $H^{\ast}_{i}$ and $\overline{H^{\ast}_{i}}$ be defined as follows:
\begin{eqnarray*}
H_{i}(\beta)&=&\alpha_{i}^{-}-\beta-\bigg(\exp\big(\beta\sup\limits_{s\in\mathbb{T}}\nu(s)\big)\bigg(
\alpha_{i}^{+}\eta_{i}^{+}\exp(\beta\eta^{+}_{i})\\
&&+\sum\limits_{j=1}^{n}D_{ij}^{+}L_{j}+\sum\limits_{j=1}^{n}(D_{ij}^\tau)^{+} L_{j}\exp(\beta\tau^{+}_{ij})+\sum\limits_{j=1}^n\bar{D}_{ij}^{+}L_{j}\sigma^{+}_{ij}\exp(\beta\sigma^{+}_{ij}) \\ &&+\sum\limits_{j=1}^n\tilde{D}_{ij}^{+}L_{j}\zeta^{+}_{ij}\exp(\beta\zeta^{+}_{ij})
\bigg)+B_{i}^{+}\bigg),\\
\overline{H_{i}}(\beta)&=&c_{i}^{-}-\beta-\big(\exp\big(\beta\sup\limits_{s\in\mathbb{T}}\nu(s)\big)c_{i}^{+}\varsigma_{i}^{+}\exp(\beta\varsigma^{+}_{i})
+E_{i}^{+}L_{i}\big),\\
H_{i}^{\ast}(\beta)&=&\alpha_{i}^{-}-\beta-\big(\alpha_{i}^{+}\exp\big(\beta\sup\limits_{s\in\mathbb{T}}\nu(s)\big)
+\alpha_{i}^{-}-\beta\big)\bigg(
\alpha_{i}^{+}\eta_{i}^{+}\exp(\beta\eta^{+}_{i})\\
&&+\sum\limits_{j=1}^{n}D_{ij}^{+}L_{j}+\sum\limits_{j=1}^{n}(D_{ij}^\tau)^{+} L_{j}\exp(\beta\tau^{+}_{ij})+\sum\limits_{j=1}^n\bar{D}_{ij}^{+}L_{j}\sigma^{+}_{ij}\exp(\beta\sigma^{+}_{ij}) \\ &&+\sum\limits_{j=1}^n\tilde{D}_{ij}^{+}L_{j}\zeta^{+}_{ij}\exp(\beta\zeta^{+}_{ij})
+B_{i}^{+}\bigg)
\end{eqnarray*}
and
\begin{eqnarray*}
\overline{H_{i}^{\ast}}(\beta)&=&c_{i}^{-}-\beta-\big(c_{i}^{+}\exp\big(\beta\sup\limits_{s\in\mathbb{T}}\nu(s)\big)+c_{i}^{-}-\beta\big)\big(c_{i}^{+}\varsigma_{i}^{+}\exp(\beta\varsigma^{+}_{i})
+E_{i}^{+}L_{i}\big).
\end{eqnarray*}
By $(H_{3})$, we get for $i=1,2,\dots,n$,
\begin{eqnarray*}
H_{i}(0)&=&\alpha_{i}^{-}-\bigg(
\alpha_{i}^{+}\eta_{i}^{+}
+\sum\limits_{j=1}^{n}D_{ij}^{+}L_{j}+\sum\limits_{j=1}^{n}(D_{ij}^\tau)^{+} L_{j}\\
&&+\sum\limits_{j=1}^n\bar{D}_{ij}^{+}L_{j}\sigma^{+}_{ij} +\sum\limits_{j=1}^n\tilde{D}_{ij}^{+}L_{j}\zeta^{+}_{ij}
+B_{i}^{+}\bigg)>0,\\
\overline{H_{i}}(0)&=&c_{i}^{-}-\big(c_{i}^{+}\varsigma_{i}^{+}
+E_{i}^{+}L_{i}\big)>0,\\
H_{i}^{\ast}(0)&=&\alpha_{i}^{-}-(\alpha_{i}^{+}
+\alpha_{i}^{-})\bigg(
\alpha_{i}^{+}\eta_{i}^{+}+\sum\limits_{j=1}^{n}D_{ij}^{+}L_{j}+\sum\limits_{j=1}^{n}(D_{ij}^\tau)^{+} L_{j} \\ &&
+\sum\limits_{j=1}^n\bar{D}_{ij}^{+}L_{j}\sigma^{+}_{ij}+\sum\limits_{j=1}^n\tilde{D}_{ij}^{+}L_{j}\zeta^{+}_{ij}
+B_{i}^{+}\bigg)>0\\
\end{eqnarray*}
and
\begin{eqnarray*}
\overline{H_{i}^{\ast}}(0)&=&c_{i}^{-}-(c_{i}^{+}+c_{i}^{-})\big(c_{i}^{+}
+E_{i}^{+}L_{i}\big)>0.
\end{eqnarray*}
Since $H_{i},\overline{H_{i}},H^{\ast}_{i}$ and $\overline{H^{\ast}_{i}}$ are continuous on $[0,+\infty)$ and
$H_{i}(\beta), \overline{H_{i}}(\beta),H^{\ast}_{i}(\beta),\overline{H^{\ast}_{i}}(\beta) \rightarrow -\infty$, as
$\beta\rightarrow +\infty$,  there exist $\xi_{i},\overline{\xi_{i}},\gamma_{i},\overline{\gamma_{i}} > 0$  such that
$H_{i}(\xi_{i})=\overline{H_{i}}(\overline{\xi_{i}})=H^{\ast}_{i}(\gamma_{i})=\overline{H^{\ast}_{i}}(\overline{\gamma_{i}})=0$ and $H_{i}(\beta)> 0$
for $\beta\in(0,\xi_{i})$, $\overline{H_{i}}(\beta)> 0$
for $\beta\in(0,\overline{\xi_{i}})$, $H^{\ast}_{i}(\beta)> 0$
for $\beta\in(0,\gamma_{i})$, $\overline{H^{\ast}_{i}}(\beta)> 0$
for $\beta\in(0,\overline{\gamma_{i}})$. Take
$a=\min\limits_{1\leq i\leq n}\big\{\xi_{i}, \overline{\xi_{i}},\gamma_{i},\overline{\gamma_{i}}\big\}$,
we have $H_{i}(a)\geq 0,\overline{H_{i}}(a)\geq 0, H^{\ast}_{i}(a)\geq 0,\overline{H^{\ast}_{i}}(a)\geq 0$. So, we can
choose a positive constant $0< \lambda <
\min\big\{a,\min\limits_{1\leq i \leq n}\{\alpha^{-}_{i},c_{i}^{-}\}\big\}$ such
that
\begin{eqnarray}\label{lam}
H_{i}(\lambda)>0,\,\,\,\overline{H_{i}}(\lambda)>0,\,\,\,
H^{\ast}_{i}(\lambda)>0,\,\,\,\overline{H^{\ast}_{i}}(\lambda)>0, \quad i=1,2,\ldots,n,
\end{eqnarray}
which imply that for $i=1,2,\ldots,n,$
\begin{eqnarray*}
&&\frac{1}{\alpha_{i}^{-}-\lambda}\bigg(\exp\big(\lambda\sup\limits_{s\in\mathbb{T}}\nu(s)\big)\bigg(
\alpha_{i}^{+}\eta_{i}^{+}+\sum\limits_{j=1}^{n}D_{ij}^{+}L_{j}+\sum\limits_{j=1}^{n}(D_{ij}^\tau)^{+} L_{j}+\sum\limits_{j=1}^n\bar{D}_{ij}^{+}L_{j}\sigma^{+}_{ij} \\ &&+\sum\limits_{j=1}^n\tilde{D}_{ij}^{+}L_{j}\zeta^{+}_{ij}\bigg)
+B_{i}^{+}\bigg)<1,
\end{eqnarray*}
\begin{eqnarray*}
\frac{1}{c_{i}^{-}-\lambda}
\big(\exp\big(\lambda\sup\limits_{s\in\mathbb{T}}\nu(s)\big)c_{i}^{+}\varsigma_{i}^{+}+E_{i}^{+}L_{i}\big)<1,
\end{eqnarray*}
\begin{eqnarray*}
&&\bigg(1+\frac{\alpha_{i}^{+}\exp(\lambda\sup\limits_{s\in
\mathbb{T}}\nu(s))}{\alpha_{i}^{-}-\lambda}\bigg)
\bigg(
\alpha_{i}^{+}\eta_{i}^{+}+\sum\limits_{j=1}^{n}D_{ij}^{+}L_{j}+\sum\limits_{j=1}^{n}(D_{ij}^\tau)^{+} L_{j}+\sum\limits_{j=1}^n\bar{D}_{ij}^{+}L_{j}\sigma^{+}_{ij} \\ &&+\sum\limits_{j=1}^n\tilde{D}_{ij}^{+}L_{j}\zeta^{+}_{ij}
+B_{i}^{+}\bigg)<1
\end{eqnarray*}
and
\begin{eqnarray*}\bigg(1+\frac{c_{i}^{+}\exp(\lambda\sup\limits_{s\in
\mathbb{T}}\nu(s))}{c_{i}^{-}-\lambda}\bigg)\big(c_{i}^{+}\varsigma_{i}^{+}+E_{i}^{+}L_{i}\big)<1.
\end{eqnarray*}

Let
\begin{eqnarray*}
M&=&\max\limits_{1\leq i \leq n}\bigg\{\frac{\alpha_{i}^{-}}{\alpha_{i}^{+}\eta_{i}^{+}+\sum\limits_{j=1}^{n}D_{ij}^{+}L_{j}+\sum\limits_{j=1}^{n}(D_{ij}^\tau)^{+} L_{j}+\sum\limits_{j=1}^n\bar{D}_{ij}^{+}L_{j}\sigma^{+}_{ij} +\sum\limits_{j=1}^n\tilde{D}_{ij}^{+}L_{j}\zeta^{+}_{ij}
+B_{i}^{+}},\\
&&\frac{c_{i}^{-}}{c_{i}^{+}\varsigma_{i}^{+}+E_{i}^{+}L_{i}}\bigg\},
\end{eqnarray*}
then by $(H_{3})$ we have $M>1$.

It is obvious that
\begin{eqnarray*}
\|Z(t)-Z^{\ast}(t)\|\leq M\hat{e}_{\ominus_{\nu}\lambda}(t,t_{0})\|\varphi\|_{0}, \qquad \forall
t\in[t_0,0]_{\mathbb{T}},\,\,\,i=1,2,\ldots,n,
\end{eqnarray*}
where $\ominus_{\nu}\lambda\in\mathcal{R}^{+}_\nu$ and $\lambda$ satisfies \eqref{lam}. We claim that
\begin{eqnarray}\label{e45}
\|Z(t)-Z^{\ast}(t)\|
\leq M\hat{e}_{\ominus_{\nu}\lambda}(t,t_{0})\|\psi\|_{0}, \qquad \forall
t\in[t_0,+\infty)_{\mathbb{T}}.
\end{eqnarray}
To prove \eqref{e45}, we show that for any $P>1$, the following inequality holds:
\begin{eqnarray}\label{e46}
\|Z(t)-Z^{\ast}(t)\|< PM\hat{e}_{\ominus_{\nu}\lambda}(t,t_{0})\|\psi\|_{0}, \quad \forall
t\in[t_{0},+\infty)_{\mathbb{T}}.
\end{eqnarray}
If \eqref{e46} is not true, then there must be some $t_{1}\in(0,+\infty)_{\mathbb{T}}$, $c\geq 1$ such that
\begin{eqnarray}\label{e48}
\|Z(t_{1})-Z^{\ast}(t_{1})\|=c PM \hat{e}_{\ominus_{\nu}\lambda}(t_{1},t_{0})\|\psi\|_{0}
\end{eqnarray}
and
\begin{eqnarray}\label{e49}
\|Z(t)-Z^{\ast}(t)\|\leq cPM \hat{e}_{\ominus_{\nu}\lambda}(t,t_{0})\|\psi\|_0,\quad t\in[t_0,t_{1}]_{\mathbb{T}}.
\end{eqnarray}
By \eqref{c2}, \eqref{e48}, \eqref{e49} and $(H_{1})$-$(H_{3})$, we obtain for $i=1,2,\ldots,n$,
\begin{eqnarray}\label{b1}
&&|u_{i}(t_{1})|\nonumber\\
&\leq&\hat{e}_{-\alpha_{i}}(t_{1},t_{0})\|\psi\|_0
+cPM\hat{e}_{\ominus_{\nu}\lambda}(t_{1},t_{0})\|\psi\|_0\int_{t_{0}}^{t_{1}}\hat{e}_{-\alpha_{i}}(t_{1},\rho(s))
\hat{e}_{\lambda}(t_{1},\rho(s))\nonumber\\
&&\times\bigg(\alpha^{+}_{i}\int_{s-\eta_{i}(s)}^{s}\hat{e}_{\lambda}(\rho(u),u) \nabla u+
\sum\limits_{j=1}^{n}D_{ij}^{+}L_{j}\hat{e}_{\lambda}(\rho(s),s)\nonumber\\
&&+\sum\limits_{j=1}^{n}(D_{ij}^\tau)^{+} L_{j}\hat{e}_{\lambda}(\rho(s),s-\tau_{ij}(s))
+\sum\limits_{j=1}^n\bar{D}_{ij}^{+}L_{j}\int_{s-\sigma_{ij}(s)}^{s}\hat{e}_{\lambda}(\rho(u),u)\nabla u \nonumber\\
&&+\sum\limits_{j=1}^n\tilde{D}_{ij}^{+}L_{j}\int_{s-\zeta_{ij}(s)}^{s}\hat{e}_{\lambda}(\rho(u),u)\nabla u
+B_i^{+}\bigg)\nabla s\nonumber\\
&\leq&\hat{e}_{-\alpha_{i}}(t_{1},t_{0})\|\psi\|_0
+cPM\hat{e}_{\ominus_{\nu}\lambda}(t_{1},t_{0})\|\psi\|_0\int_{t_{0}}^{t_{1}}
\hat{e}_{-\alpha_{i}\oplus_{\nu}\lambda}(t_{1},\rho(s))\nonumber\\
&&\times\bigg(\alpha^{+}_{i}\eta^{+}_{i}\hat{e}_{\lambda}(\rho(s),s-\eta_{i}(s))+
\sum\limits_{j=1}^{n}D_{ij}^{+}L_{j}\hat{e}_{\lambda}(\rho(s),s)\nonumber\\
&&+\sum\limits_{j=1}^{n}(D_{ij}^\tau)^{+} L_{j}\hat{e}_{\lambda}(\rho(s),s-\tau_{ij}(s))
+\sum\limits_{j=1}^n\bar{D}_{ij}^{+}L_{j}\sigma_{ij}^+\hat{e}_{\lambda}(\rho(s),s-\sigma_{ij}(s))\nonumber \\
 &&+\sum\limits_{j=1}^n\tilde{D}_{ij}^{+}L_{j}\zeta_{ij}^+\hat{e}_{\lambda}(\rho(s),s-\zeta_{ij}(s))
+B_i^{+}\bigg)\nabla s\nonumber\\
  &\leq&\hat{e}_{-\alpha_{i}}(t_{1},t_{0})\|\psi\|_0
+cPM\hat{e}_{\ominus_{\nu}\lambda}(t_{1},t_{0})\|\psi\|_0\int_{t_{0}}^{t_{1}}
\hat{e}_{-\alpha_{i}\oplus_{\nu}\lambda}(t_{1},\rho(s))\nonumber\\
&&\times\bigg(\alpha_{i}^{+}\eta_{i}^{+}\exp\big[\lambda(\eta^{+}_{i}+\sup\limits_{s\in\mathbb{T}}\nu(s))\big]+
\sum\limits_{j=1}^{n}D_{ij}^{+}L_{j}\exp\big(\lambda\sup\limits_{s\in\mathbb{T}}\nu(s)\big)\nonumber\\
&&+\sum\limits_{j=1}^{n}(D_{ij}^\tau)^{+} L_{j}\exp\big[\lambda(\tau^{+}_{ij}+\sup\limits_{s\in\mathbb{T}}\nu(s))\big]
+\sum\limits_{j=1}^n\bar{D}_{ij}^{+}L_{j}\sigma^{+}_{ij}\exp\big[\lambda(\sigma^{+}_{ij}+\sup\limits_{s\in\mathbb{T}}\nu(s))\big]\nonumber \\ &&+\sum\limits_{j=1}^n\tilde{D}_{ij}^{+}L_{j}\zeta^{+}_{ij}\exp\big[\lambda(\zeta^{+}_{ij}+\sup\limits_{s\in\mathbb{T}}\nu(s))\big]
+B_i^{+}\bigg)\nabla s\nonumber\\
&<&cPM\hat{e}_{\ominus_{\nu}\lambda}(t_{1},t_{0})\|\psi\|_0
\bigg\{\frac{1}{M}\hat{e}_{-\alpha_{i}\oplus_{\nu}\lambda}(t_{1},t_{0})+\bigg[\exp(\lambda\sup\limits_{s\in\mathbb{T}}\nu(s))\nonumber\\
&&\times\bigg(\alpha_{i}^{+}\eta_{i}^{+}\exp(\lambda\eta^{+}_{i})+
\sum\limits_{j=1}^{n}D_{ij}^{+}L_{j}+\sum\limits_{j=1}^{n}(D_{ij}^\tau)^{+} L_{j}\exp(\lambda\tau^{+}_{ij})\nonumber\\
&&+\sum\limits_{j=1}^n\bar{D}_{ij}^{+}L_{j}\sigma^{+}_{ij}\exp(\lambda\sigma^{+}_{ij}) +\sum\limits_{j=1}^n\tilde{D}_{ij}^{+}L_{j}\zeta^{+}_{ij}\exp(\lambda\zeta^{+}_{ij})\bigg)\nonumber\\
&&+B_i^{+}\bigg]\frac{1-\hat{e}_{-\alpha_{i}\oplus_{\nu}\lambda}(t_{1},t_{0})}{\alpha_{i}^{-}-\lambda}
\bigg\}\nonumber\\
&\leq&cPM \hat{e}_{\ominus_{\nu}\lambda}(t_{1},t_{0})\|\psi\|_0\bigg\{
\bigg[\frac{1}{M}-\frac{1}{\alpha_{i}^{-}-\lambda}\bigg(\exp\big(\lambda\sup\limits_{s\in\mathbb{T}}\nu(s)\big)\bigg(
\alpha_{i}^{+}\eta_{i}^{+}\exp(\lambda\eta^{+}_{i})\nonumber\\
&&+\sum\limits_{j=1}^{n}D_{ij}^{+}L_{j}+\sum\limits_{j=1}^{n}(D_{ij}^\tau)^{+} L_{j}\exp(\lambda\tau^{+}_{ij})+\sum\limits_{j=1}^n\bar{D}_{ij}^{+}L_{j}\sigma^{+}_{ij}\exp(\lambda\sigma^{+}_{ij}) \nonumber\\ &&+\sum\limits_{j=1}^n\tilde{D}_{ij}^{+}L_{j}\zeta^{+}_{ij}\exp(\lambda\zeta^{+}_{ij})
\bigg)+B_{i}^{+}\bigg)\bigg]\hat{e}_{-\alpha_{i}\oplus_{\nu}\lambda}(t_{1},t_{0})\nonumber\\
&&+\frac{1}{\alpha_{i}^{-}-\lambda}
\bigg(\exp\big(\lambda\sup\limits_{s\in\mathbb{T}}\nu(s)\big)\bigg(\alpha_{i}^{+}\eta_{i}^{+}\exp(\lambda\eta^{+}_{i})
+\sum\limits_{j=1}^{n}D_{ij}^{+}L_{j}\nonumber\\
&&+\sum\limits_{j=1}^{n}(D_{ij}^\tau)^{+} L_{j}\exp(\lambda\tau^{+}_{ij})
+\sum\limits_{j=1}^n\bar{D}_{ij}^{+}L_{j}\sigma^{+}_{ij}\exp(\lambda\sigma^{+}_{ij}) \nonumber\\
&&+\sum\limits_{j=1}^n\tilde{D}_{ij}^{+}L_{j}\zeta^{+}_{ij}\exp(\lambda\zeta^{+}_{ij})\bigg)+B_{i}^{+}\bigg)
\bigg\}\nonumber\\
&\leq&cPM \hat{e}_{\ominus_{\nu}\lambda}(t_{1},t_{0})\|\psi\|_0
\end{eqnarray}
and
\begin{eqnarray}\label{b2}
&&|v_{i}(t_{1})|\nonumber\\
&\leq&
\hat{e}_{-c_{i}}(t_{1},t_{0})\|\psi\|_0
+cPM\hat{e}_{\ominus_{\nu}\lambda}(t_{1},t_{0})\|\psi\|_0\int_{t_{0}}^{t_{1}}\hat{e}_{-c_{i}}(t_{1},\rho(s))
\hat{e}_{\lambda}(t_{1},\rho(s))\nonumber\\
&&\times\bigg(c^{+}_{i}\int_{s-\varsigma_{i}(s)}^{s}\hat{e}_{\lambda}(\rho(u),u) \nabla u+E_{i}^{+}L_{i}\bigg)\nabla s\nonumber\\
&\leq&\hat{e}_{-c_{i}}(t_{1},t_{0})\|\psi\|_0
+cPM\hat{e}_{\ominus_{\nu}\lambda}(t_{1},t_{0})\|\psi\|_0\int_{t_{0}}^{t_{1}}
\hat{e}_{-c_{i}\oplus_{\nu}\lambda}(t_{1},\rho(s))\nonumber\\
&&\times\bigg(c^{+}_{i}\varsigma^{+}_{i}\hat{e}_{\lambda}(\rho(s),s-\varsigma_{i}(s))+E_{i}^{+}L_{i}\bigg)\nabla s\nonumber\\
  &\leq&\hat{e}_{-c_{i}}(t_{1},t_{0})\|\psi\|_0
+cPM\hat{e}_{\ominus_{\nu}\lambda}(t_{1},t_{0})\|\psi\|_0\int_{t_{0}}^{t_{1}}
\hat{e}_{-c_{i}\oplus_{\nu}\lambda}(t_{1},\rho(s))\nonumber\\
&&\times\bigg(c_{i}^{+}\varsigma_{i}^{+}\exp\big[\lambda(\varsigma^{+}_{i}+\sup\limits_{s\in\mathbb{T}}\nu(s))\big]
+E_{i}^{+}L_{i}\bigg)\nabla s\nonumber\\
&<&
cPM\hat{e}_{\ominus_{\nu}\lambda}(t_{1},t_{0})\|\psi\|_0
\bigg\{\frac{1}{M}\hat{e}_{-c_{i}\oplus_{\nu}\lambda}(t_{1},t_{0})\nonumber\\
&&+\big[\exp(\lambda\sup\limits_{s\in\mathbb{T}}\nu(s))c_{i}^{+}\varsigma_{i}^{+}\exp(\lambda\varsigma^{+}_{i})+E_{i}^{+}L_{i}\big]
\frac{1-\hat{e}_{-c_{i}\oplus_{\nu}\lambda}(t_{1},t_{0})}{c_{i}^{-}-\lambda}\bigg\}\nonumber\\
&\leq&cPM \hat{e}_{\ominus_{\nu}\lambda}(t_{1},t_{0})\|\psi\|_0\bigg\{
\bigg[\frac{1}{M}-\frac{1}{c_{i}^{-}-\lambda}
\big(\exp\big(\lambda\sup\limits_{s\in\mathbb{T}}\nu(s)\big)c_{i}^{+}\varsigma_{i}^{+}\exp(\lambda\varsigma^{+}_{i})\nonumber\\
&&+E_{i}^{+}L_{i}\big)\bigg]\hat{e}_{-c_{i}\oplus_{\nu}\lambda}(t_{1},t_{0})
+\frac{1}{c_{i}^{-}-\lambda}
\big(\exp\big(\lambda\sup\limits_{s\in\mathbb{T}}\nu(s)\big)c_{i}^{+}\varsigma_{i}^{+}\exp(\lambda\varsigma^{+}_{i})+E_{i}^{+}L_{i}\big)\bigg\}\nonumber\\
&\leq&cPM \hat{e}_{\ominus_{\nu}\lambda}(t_{1},t_{0})\|\psi\|_0.
\end{eqnarray}

Similarly, in view of \eqref{c2}, we have for $i=1,2,\ldots,n$,
\begin{eqnarray}\label{b3}
&&|u_{i}^{\nabla}(t_{1})|\nonumber\\
&\leq&
\alpha_{i}^{+} \hat{e}_{-\alpha_{i}}(t_{1},t_{0})\|\psi\|_0+cP
M\hat{e}_{\ominus_{\nu}\lambda}(t_{1},t_{0})\|\psi\|_0
\bigg(\alpha_{i}^{+}\int_{t_{1}-\eta_{i}(t_{1})}^{t_{1}}\hat{e}_{\lambda}(t_{1},u)
\nabla u\nonumber\\
&&+\sum_{j=1}^{n}D_{ij}^{+}L_{j}\hat{e}_{\lambda}(t_1,t_1)
+\sum_{j=1}^{n}(D_{ij}^{\tau})^{+}L_{j}\hat{e}_{\lambda}(t_1,t_1-\tau_{ij}(t_1))\nonumber\\
&&+\sum_{j=1}^{n}\bar{D}_{ij}^{+}L_{j}\int_{t_1-\sigma_{ij}(t_1)}^{t_1}\hat{e}_{\lambda}(\rho(u),u)\nabla u\nonumber\\
&&+\sum_{j=1}^{n}\tilde{D}_{ij}^{+}L_{j}\int_{t_1-\zeta_{ij}(t_1)}^{t_1}\hat{e}_{\lambda}(\rho(u),u)\nabla u+B_{i}^{+}
\bigg)\nonumber\\
&&+ \alpha_{i}^{+}cP M\hat{e}_{\ominus_{\nu}\lambda}(t_{1},t_{0})\|\psi\|_0\int_{t_{0}}^{t_{1}}\hat{e}_{-\alpha_{i}}(t_{1},
\rho(s))\hat{e}_{\lambda}(t_{1},\rho(s))\nonumber\\
&&\times\bigg\{
\alpha_{i}^{+}\int_{s-\eta_{i}(s)}^{s}\hat{e}_{\lambda}(\rho(u),u) \nabla u
+\sum_{j=1}^{n}D_{ij}^{+}L_{j}\hat{e}_{\lambda}(\rho(s),s)\nonumber\\
&&+\sum_{j=1}^{n}(D_{ij}^{\tau})^{+}L_{j}\hat{e}_{\lambda}(\rho(s),s-\tau_{ij}(t))
+\sum_{j=1}^{n}\bar{D}_{ij}^{+}L_{j}\int_{s-\sigma_{ij}(s)}^{s}\hat{e}_{\lambda}(\rho(u),u)\nabla u\nonumber\\
&&+\sum_{j=1}^{n}\tilde{D}_{ij}^{+}L_{j}\int_{s-\zeta_{ij}(s)}^{s}\hat{e}_{\lambda}(\rho(u),u)\nabla u+B_{i}^{+}\bigg\}
\nabla s\nonumber\\
&\leq&\alpha_{i}^{+} e_{-\alpha_{i}}(t_{1},t_{0})\|\psi\|_0+cP
M\hat{e}_{\ominus_{\nu}\lambda}(t_{1},t_{0})\|\psi\|_0
\bigg(\alpha_{i}^{+}\eta_{i}^{+}\exp(\lambda\eta_{i}^{+})
+\sum_{j=1}^{n}D_{ij}^{+}L_{j}\nonumber\\
&&+\sum_{j=1}^{n}(D_{ij}^{\tau})^{+}L_{j}\exp(\lambda\tau_{ij}^{+})
+\sum_{j=1}^{n}\bar{D}_{ij}^{+}L_{j}\sigma_{ij}^{+}\exp(\lambda\sigma_{ij}^{+})\nonumber\\
&&+\sum_{j=1}^{n}\tilde{D}_{ij}^{+}L_{j}\zeta_{ij}^{+}\exp(\lambda\zeta_{ij}^{+})+B_{i}^{+}
\bigg)\nonumber\\
&&\times\bigg(1+\alpha_{i}^{+}\exp(\lambda\sup\limits_{s\in\mathbb{T}}\nu(s))
\int_{t_{0}}^{t_{1}}\hat{e}_{-\alpha_{i}\oplus\lambda}(t_{1},\rho(s))
\nabla s\bigg)\nonumber\\
&\leq& cP M\hat{e}_{\ominus_{\nu}\lambda}(t_{1},t_{0})\|\psi\|_0
\bigg\{\bigg[\frac{1}{M}-\frac{\exp\big(\lambda\sup\limits_{s\in\mathbb{T}}\nu(s)\big)}{\alpha_{i}^{-}-\lambda}\bigg(
\alpha_{i}^{+}\eta_{i}^{+}\exp(\lambda\eta^{+}_{i})\nonumber\\
&&+\sum\limits_{j=1}^{n}D_{ij}^{+}L_{j}+\sum\limits_{j=1}^{n}(D_{ij}^\tau)^{+} L_{j}\exp(\lambda\tau^{+}_{ij})+\sum\limits_{j=1}^n\bar{D}_{ij}^{+}L_{j}\sigma^{+}_{ij}\exp(\lambda\sigma^{+}_{ij}) \nonumber\\ &&+\sum\limits_{j=1}^n\tilde{D}_{ij}^{+}L_{j}\zeta^{+}_{ij}\exp(\lambda\zeta^{+}_{ij})
+B_{i}^{+}\bigg)\bigg]\alpha_{i}^{+}\hat{e}_{-\alpha_{i}\oplus_{\nu}\lambda}(t_{1},t_{0})\nonumber\\
&&+\bigg(1+\frac{\alpha_{i}^{+}\exp\big(\lambda\sup\limits_{s\in\mathbb{T}}\nu(s)\big)}{\alpha_{i}^{-}-\lambda}\bigg)
\bigg(\alpha_{i}^{+}\eta_{i}^{+}\exp(\lambda\eta^{+}_{i})
+\sum\limits_{j=1}^{n}D_{ij}^{+}L_{j}\nonumber\\
&&+\sum\limits_{j=1}^{n}(D_{ij}^\tau)^{+} L_{j}\exp(\lambda\tau^{+}_{ij})
+\sum\limits_{j=1}^n\bar{D}_{ij}^{+}L_{j}\sigma^{+}_{ij}\exp(\lambda\sigma^{+}_{ij})\nonumber\\
&&+\sum\limits_{j=1}^n\tilde{D}_{ij}^{+}L_{j}\zeta^{+}_{ij}\exp(\lambda\zeta^{+}_{ij})+B_{i}^{+}\bigg)
\bigg\}\nonumber\\
&<&cP M\hat{e}_{\ominus_{\nu}\lambda}(t_{1},t_{0})\|\psi\|_0
\end{eqnarray}
and
\begin{eqnarray}\label{b4}
&&|v_{i}^{\nabla}(t_{1})|\nonumber\\
&\leq&
c_{i}^{+} \hat{e}_{-c_{i}}(t_{1},t_{0})\|\psi\|_0+cP
M\hat{e}_{\ominus_{\nu}\lambda}(t_{1},t_{0})\|\psi\|_0
\bigg(c_{i}^{+}\int_{t_{1}-\varsigma_{i}(t_{1})}^{t_{1}}\hat{e}_{\lambda}(t_{1},u)
\nabla u\nonumber\\
&&+E_{i}^{+}L_{i}
\bigg)
+c_{i}^{+}cP M\hat{e}_{\ominus_{\nu}\lambda}(t_{1},t_{0})\|\psi\|_0\int_{t_{0}}^{t_{1}}\hat{e}_{-c_{i}}(t_{1},
\rho(s))\hat{e}_{\lambda}(t_{1},\rho(s))\nonumber\\
&&\times\bigg\{
c_{i}^{+}\int_{s-\varsigma_{i}(s)}^{s}\hat{e}_{\lambda}(\rho(u),u) \nabla u
+E_{i}^{+}L_{i}\bigg\}
\nabla s\nonumber\\
&\leq&c_{i}^{+} e_{-c_{i}}(t_{1},t_{0})\|\psi\|_0+cP
M\hat{e}_{\ominus_{\nu}\lambda}(t_{1},t_{0})\|\psi\|_0
\bigg(c_{i}^{+}\varsigma_{i}^{+}\exp(\lambda\varsigma_{i}^{+})
+E_{i}^{+}L_{i}
\bigg)\nonumber\\
&&\times\bigg(1+c_{i}^{+}\exp(\lambda\sup\limits_{s\in\mathbb{T}}\nu(s))
\int_{t_{0}}^{t_{1}}\hat{e}_{-c_{i}\oplus\lambda}(t_{1},\rho(s))
\nabla s\bigg)\nonumber\\
&\leq&
cPM \hat{e}_{\ominus_{\nu}\lambda}(t_{1},t_{0})\|\psi\|_0\bigg\{
\bigg[\frac{1}{M}-\frac{\exp\big(\lambda\sup\limits_{s\in\mathbb{T}}\nu(s)\big)}{c_{i}^{-}-\lambda}
\big(c_{i}^{+}\varsigma_{i}^{+}\exp(\lambda\varsigma^{+}_{i})\nonumber\\
&&+E_{i}^{+}L_{i}\big)\bigg]c_{i}^{+}\hat{e}_{-c_{i}\oplus_{\nu}\lambda}(t_{1},t_{0})
+\bigg(1+\frac{c_{i}^{+}\exp\big(\lambda\sup\limits_{s\in\mathbb{T}}\nu(s)\big)}{c_{i}^{-}-\lambda}\bigg)
\big(c_{i}^{+}\varsigma_{i}^{+}\exp(\lambda\varsigma^{+}_{i})+E_{i}^{+}L_{i}\big)\bigg\}\nonumber\\
&<&cP M\hat{e}_{\ominus_{\nu}\lambda}(t_{1},t_{0})\|\psi\|_0.
\end{eqnarray}
In view of \eqref{b1}-\eqref{b4}, we have
\begin{eqnarray*}
\|Z(t_1)-Z^{\ast}(t_1)\|<cPM \hat{e}_{\ominus_{\nu}\lambda}(t_{1},t_{0})\|\psi\|_{0},
\end{eqnarray*}
which contradicts \eqref{e48}, and so \eqref{e46} holds. Letting $p\rightarrow
1$, then \eqref{e45} holds. Hence, the almost periodic solution of system
\eqref{a1} is globally exponentially stable. The proof is complete.
\end{proof}

\begin{remark}
Since  conditions $(A_1)$-$(A_3)$ are independent of the backwards graininess function  of the time scale, according to Theorems \ref{thm31} and \ref{thm41},  both the continuous time and the discrete time cases of system \eqref{a1} share the same dynamical behavior. 
\end{remark}

\section{ An example}
 \setcounter{equation}{0}
 \indent

In this section, we present an example to illustrate the feasibility
of our results obtained in previous sections.

\begin{example} In system \eqref{a1}, suppose $\sup\limits_{t\in\mathbb{T}}\nu(t)<+\infty$,
let $i,j=2$ and take coefficients as follows:
\[
\left[
                                                                \begin{array}{c}
                                                                  f_1(x) \\
                                                                  f_2(x) \\
                                                                \end{array}
                                                              \right]=\left[
                                                                        \begin{array}{c}
                                                                          \sin \frac{x}{2} \\
                                                                          \sin \frac{x}{2} \\
                                                                        \end{array}
                                                                      \right],
\alpha_1(t)=0.895+0.005\sin \sqrt 7 t,\,\alpha_{2}(t)=0.79+0.01\cos \sqrt 11 t,\]
\[\left[
    \begin{array}{cc}
      D_{11}(t) & D_{12}(t) \\
      D_{21}(t) & D_{22}(t) \\
    \end{array}
  \right]=
  \left[
    \begin{array}{cc}
      D_{11}^{\tau}(t) & D_{12}^{\tau}(t) \\
      D_{21}^{\tau}(t) & D_{22}^{\tau}(t) \\
    \end{array}
  \right]=\left[
    \begin{array}{cc}
      \bar{D}_{11}(t) & \bar{D}_{12}(t) \\
      \bar{D}_{21}(t) & \bar{D}_{22}(t) \\
    \end{array}
  \right]\]
    \[=\left[
    \begin{array}{cc}
      \tilde{D}_{11}(t) & \tilde{D}_{12}(t) \\
      \tilde{D}_{21}(t) & \tilde{D}_{22}(t) \\
    \end{array}
      \right]      =\left[
        \begin{array}{cc}
          \frac{\sin t}{20} & \frac{\sin t}{20} \\
          \frac{\cos t}{20} & \frac{\cos t}{20} \\
        \end{array}
      \right]
      ,\]
\[B_{1}(t)=\frac{\sin \sqrt 2t}{\pi e^{2\pi}},B_{2}(t)=\frac{\cos t}{\pi e^{2\pi}},
c_1(t)=0.285+0.005\sin \sqrt 5t,\,c_{2}(t)=0.275+0.005\cos \sqrt 3t,\]
\[
E_1(t)=0.21\sin t,\,E_{2}(t)=0.16\cos \sqrt 3t,\,
I_1(t)=0.08\sin \sqrt 7 t,\,I_{2}(t)=0.1\cos t,\]
\[
J_1(t)=0.01\sin \sqrt 2t,\,J_{2}(t)=0.02\cos \sqrt 3t,\,
\eta_1(t)=e^{-5|\cos( \pi t+\frac{3\pi}{2})|},\,\eta_2(t)=e^{-4|\cos(\pi t+\frac{\pi}{2})|},\]
\[
\sigma_{11}(t)=e^{-4|\sin\pi t|},\,\sigma_{12}(t)=e^{-5|\cos(\pi t+\frac{3\pi}{2})|},\,
\sigma_{21}(t)=e^{-6|\cos(\pi t-\frac{3\pi}{2})|},\,\sigma_{22}(t)=e^{-4|\sin 3\pi t|},\]
\[
\zeta_{11}(t)=e^{-7|\sin 2\pi t|},\,\zeta_{12}(t)=e^{-5|\sin 5\pi t|},\,
\zeta_{21}(t)=e^{-4|\cos(\pi t+\frac{5\pi}{2})|},\,\zeta_{22}(t)=e^{-5|\cos(\pi t+\frac{\pi}{2})|},\]
\[
\varsigma_1(t)=e^{-4|\cos(\pi t+\frac{3\pi}{2})|},\,\varsigma_2(t)=e^{-7|\sin 3\pi t|}.\]
\end{example}

By calculating, we have
\[
\alpha_{1}^+=0.9, \alpha_{1}^-=0.89, \alpha_{2}^+=0.8, \alpha_{2}^-=0.78, c_{1}^+=0.29,
c_{1}^-=0.28,c_{2}^+=0.28,
c_{2}^-=0.27,
\]
\[\left[
    \begin{array}{cc}
      D_{11}^+& D_{12}^+ \\
      D_{21}^+ & D_{22}^+ \\
    \end{array}
  \right]=
  \left[
    \begin{array}{cc}
      D_{11}^{\tau+} & D_{12}^{\tau+} \\
      D_{21}^{\tau+} & D_{22}^{\tau+} \\
    \end{array}
  \right]=\left[
    \begin{array}{cc}
      \bar{D}_{11}^+ & \bar{D}_{12}^+ \\
      \bar{D}_{21}^+ & \bar{D}_{22}^+ \\
    \end{array}
  \right]=\left[
    \begin{array}{cc}
      \tilde{D}_{11}^+ & \tilde{D}_{12}^+ \\
      \tilde{D}_{21}^+ & \tilde{D}_{22}^+ \\
    \end{array}
      \right]      =\left[
        \begin{array}{cc}
          \frac{1}{20} & \frac{1}{20} \\
          \frac{1}{20} & \frac{1}{20} \\
        \end{array}
      \right]
      ,\]
\[I_{1}^+=0.08,\,I_{2}^+=0.1,\,J_{1}^+=0.01,\,J_{2}^+=0.02,\,B_{1}^+=B_{2}^+=\frac{1}{\pi e^{2\pi}},\,E_{1}^+=0.21,\,E_{2}^+=0.16,
\]
\[
\eta_{1}^+=0.06, \eta_{2}^+=0.05, \sigma_{11}^+=0.08,\sigma_{12}^+=0.07,\sigma_{21}^+=0.04,\sigma_{22}^+=0.02,\]\[
\zeta_{11}^+=0.06,\zeta_{12}^+=0.05,\zeta_{21}^+=0.02,\zeta_{22}^+=0.03,\varsigma_{1}^+=0.4,\varsigma_{2}^+=0.5.
\]
Take $L_1=L_2=1$ and $r=0.45,$ we have
\begin{eqnarray*}P_1&=&\alpha_{1}^{+}\eta_{1}^{+}r+\sum\limits_{j=1}^{2}D_{1j}^{+}(L_{j}r+|f_j(0)|)
+\sum\limits_{j=1}^{2}(D_{1j}^\tau)^{+}(L_{j}r+|f_j(0)|)\\
&&+\sum\limits_{j=1}^2\bar{D}_{1j}^{+}\sigma_{1j}^{+}(L_{j}r +|f_j(0)|)+\sum\limits_{j=1}^2\tilde{D}_{1j}^{+}\zeta_{1j}^{+}(L_{j}r+|f_j(0)|)
+B_1^{+}r+I_1^+\\&=&
0.9\times0.06\times0.45+1/20\times(1\times0.45+0)+1/20\times(1\times0.45+0)\\& &+1/20\times(1\times0.45+0)+1/20\times(1\times0.45+0)+1/20\times(1\times0.45+0)\\& &+1/20\times(1\times0.45+0)+1/(\pi e^{2\pi})\times 0.45+0.08\\
&=&0.2004,\\
P_2&=&\alpha_{2}^{+}\eta_{2}^{+}r+\sum\limits_{j=1}^{2}D_{2j}^{+}(L_{j}r+|f_j(0)|)
+\sum\limits_{j=1}^{2}(D_{2j}^\tau)^{+}(L_{j}r+|f_j(0)|)\\
&&+\sum\limits_{j=1}^2\bar{D}_{2j}^{+}\sigma_{2j}^{+}(L_{j}r +|f_j(0)|)+\sum\limits_{j=1}^2\tilde{D}_{2j}^{+}\zeta_{2j}^{+}(L_{j}r+|f_j(0)|)
+B_2^{+}r+I_2^+\\&=&
0.8\times0.05\times0.45+1/20\times(1\times0.45+0)+1/20\times(1\times0.45+0)\\& &+1/20\times(1\times0.45+0)+1/20\times(1\times0.45+0)+1/20\times(1\times0.45+0)\\& &+1/20\times(1\times0.45+0)+1/(\pi e^{2\pi})\times 0.45+0.1\\
&=&0.2107,
\\
Q_1& =&c_{1}^{+}\varsigma_{1}^{+}r
  +E_1^{+}(L_{1}r+|f_1(0)|)+J_1^+=0.29\times0.04\times0.45+0.21\times(0.45+0)\\& &+0.01= 0.1097,\\
  Q_2& =&c_{2}^{+}\varsigma_{2}^{+}r
  +E_2^{+}(L_{2}r+|f_2(0)|)+J_2^+=0.28\times0.05\times0.45+0.16\times(0.45+0)\\& &+0.02=0.1208,\\
  \overline{P_1}&=&\alpha_{1}^{+}\eta_{1}^{+}
  +\sum\limits_{j=1}^{2}D_{1j}^{+}L_{j} +\sum\limits_{j=1}^{2}(D_{1j}^\tau)^{+}L_{j} +\sum\limits_{j=1}^2\bar{D}_{1j}^{+}\sigma_{1j}^{+}L_{j}  +\sum\limits_{j=1}^2\tilde{D}_{1j}^{+}\zeta_{1j}^{+}L_{j} +B_1^{+}\\&=&
  0.9\times0.06+1/20+1/20+1/20\times0.08+1/20\times0.07\\& &+1/20\times0.06+1/20\times0.05+1/(\pi e^{2\pi})=0.1676,\\
  \overline{P_2}&=&\alpha_{2}^{+}\eta_{2}^{+}
  +\sum\limits_{j=1}^{2}D_{2j}^{+}L_{j} +\sum\limits_{j=1}^{2}(D_{2j}^\tau)^{+}L_{j} +\sum\limits_{j=1}^2\bar{D}_{2j}^{+}\sigma_{2j}^{+}L_{j}  +\sum\limits_{j=1}^2\tilde{D}_{2j}^{+}\zeta_{2j}^{+}L_{j} +B_2^{+}\\&=&
  0.8\times0.05+1/20+1/20+1/20\times0.04+1/20\times0.02\\& &+1/20\times0.02+1/20\times0.03+1/(\pi e^{2\pi})=0.1471, \\
  \overline{Q_1}&=&c_{1}^{+}\varsigma_{1}^{+}
  +E_1^{+}L_{1}=
  0.29\times0.04+0.21\times1=0.2216,\\
  \overline{Q_2}&=&c_{2}^{+}\varsigma_{2}^{+}
  +E_2^{+}L_{2}=
  0.28\times0.05+0.16\times1=0.224.\\
\end{eqnarray*}
Obviously, conditions $(H_1)$ and $(H_2)$ hold. Since
\begin{eqnarray*}
&&\max\limits_{1\leq i\leq n}\bigg\{\frac{P_{i}}{\alpha_{i}^{-}},\bigg(1+\frac{\alpha_i^+}{\alpha_{i}^{-}}\bigg)P_{i},\frac{Q_{i}}{c_{i}^{-}},\bigg(1+\frac{c_i^+}{c_{i}^{-}}\bigg)Q_{i}
\bigg\}\leq r,
\\
&& \max\limits_{1\leq i\leq 2}\bigg\{\frac{\overline{P_{i}}}{\alpha_{i}^{-}},\bigg(1+\frac{\alpha_i^+}{\alpha_{i}^{-}}\bigg)\overline{P_{i}},\frac{\overline{Q_{i}}}{c_{i}^{-}},\bigg(1+\frac{c_i^+}{c_{i}^{-}}\bigg)\overline{Q_{i}}
\bigg\}< 1,
\end{eqnarray*}
that is,
\begin{eqnarray*}
&&\max\limits_{1\leq i\leq 2}\{0.2252,0.4031,0.2702,0.4269,0.3919,0.4474,0.2234,0.2461\}\\&&=0.4474\leq r=0.45,
\\
&&\max\limits_{1\leq i\leq 2}\{0.1883,0.1886,0.3371,0.2980,0.7914,0.8296,0.4511,0.4563\}\\&&=0.8296\leq 1,
\end{eqnarray*}
so, condition $(H_3)$ holds. Thus,
all the conditions in Theorem \ref{thm41} are satisfied. Therefore, according to Theorem \ref{thm41}, system \eqref{a1} has a unique almost periodic solution which is globally exponentially stable. Especially,  for both the discrete time and continuous time cases,  system \eqref{a1} has a unique almost periodic solution which is globally exponentially stable (see Figures 1-4).
\begin{figure}
  \centering
  \includegraphics[width=13cm,height=9cm]{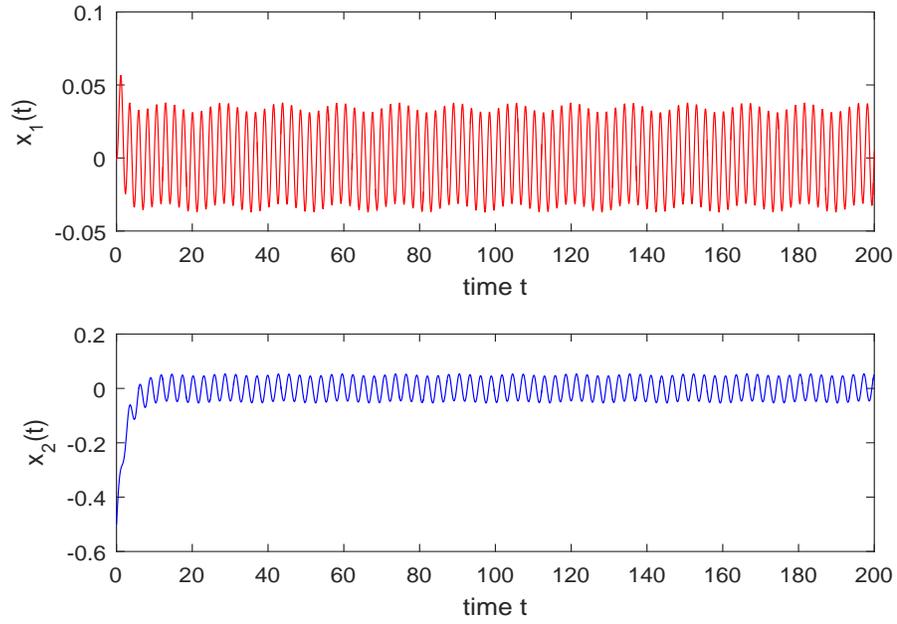}
  \caption{Continuous situation $\mathbb{T}=\mathbb{R}: x_1,x_2$ with $t$.}
\end{figure}
\begin{figure}
\centering
  \includegraphics[width=13cm,height=9cm]{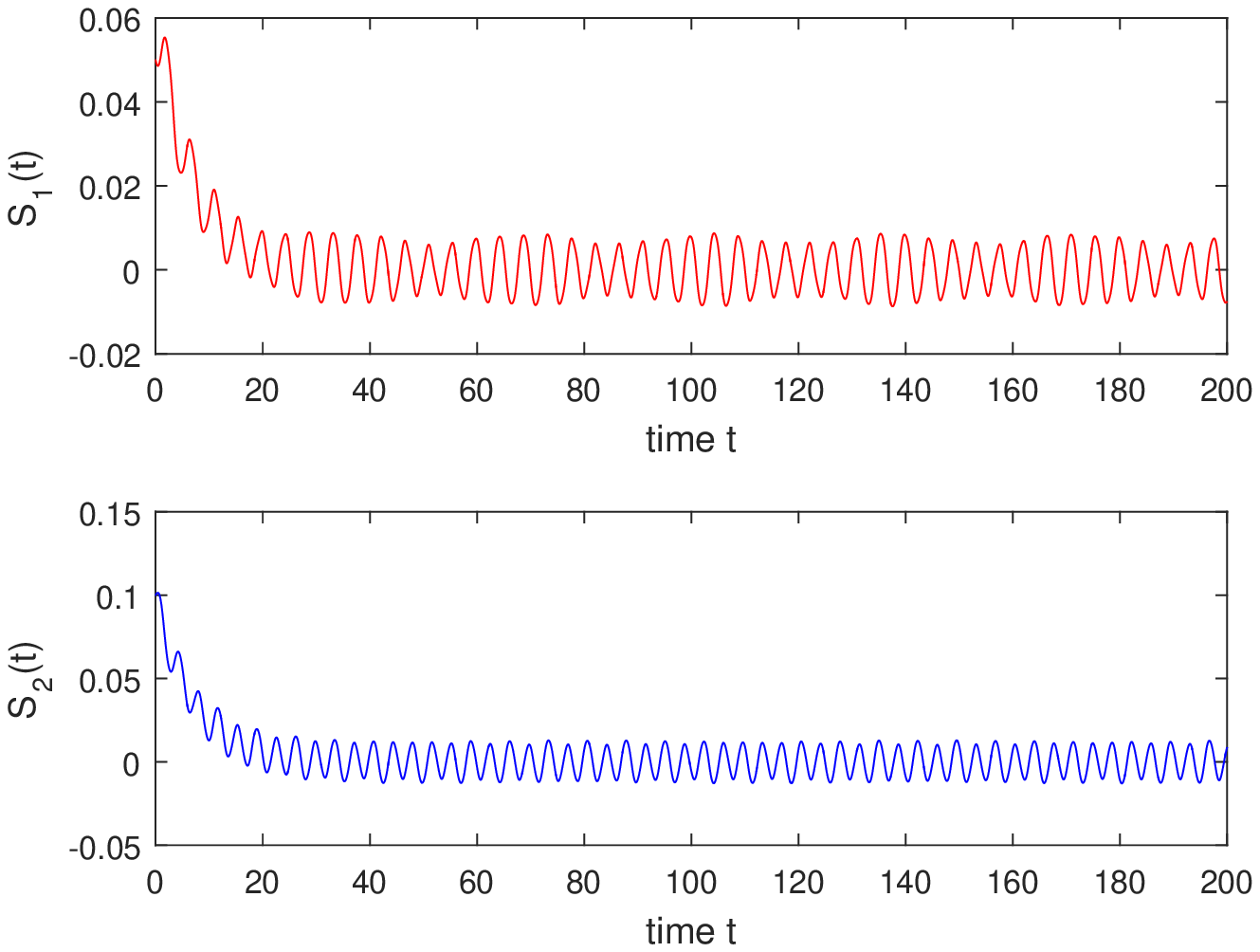}
  \caption{Continuous situation $\mathbb{T}=\mathbb{R}: S_1,S_2$ with $t$.}
\end{figure}
\begin{figure}
\centering
  \includegraphics[width=13cm,height=9cm]{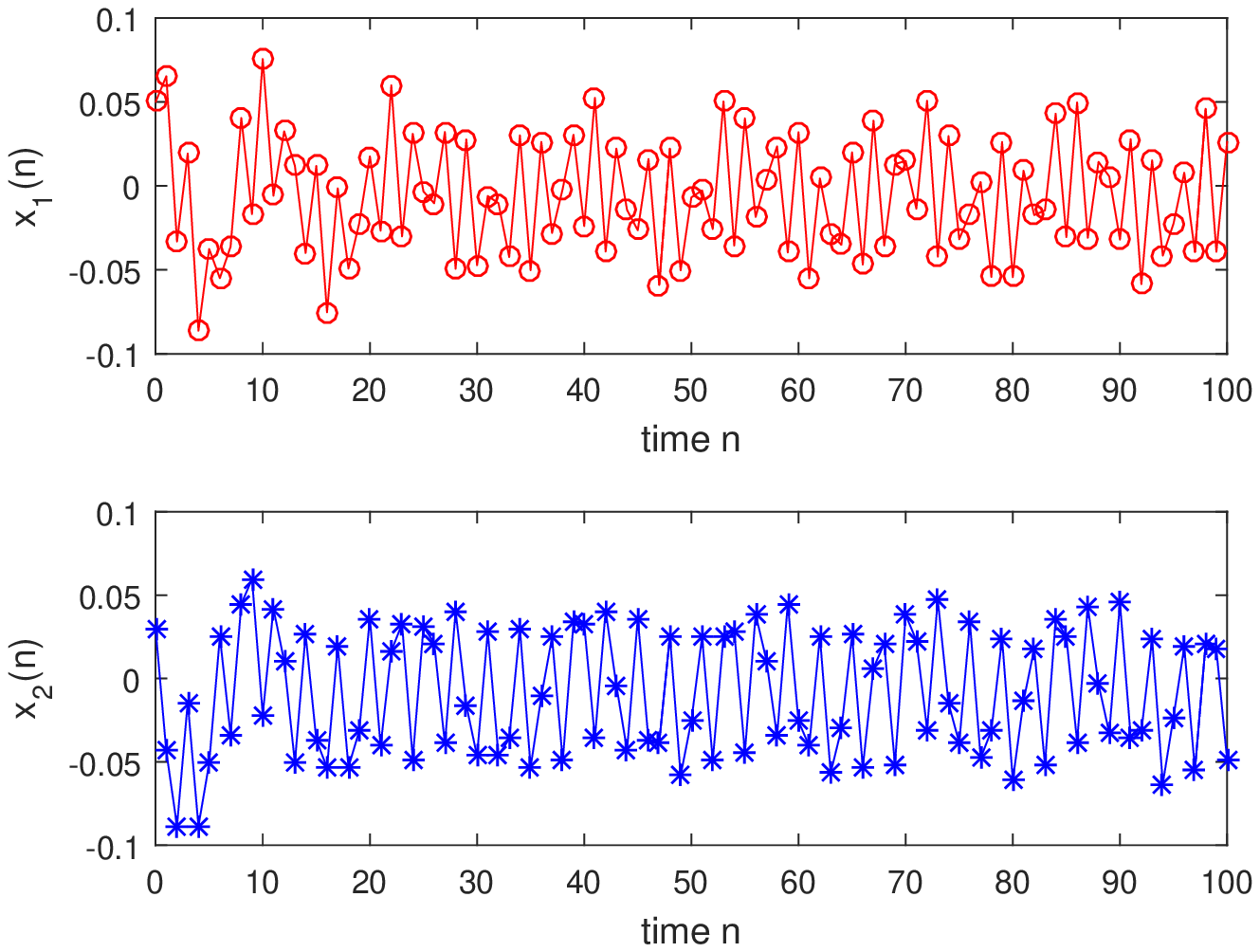}%
  \caption{Discrete situation  $\mathbb{T}=\mathbb{Z}: x_1,x_2$ with $n$.}
\end{figure}

\begin{figure}
  \centering
  \includegraphics[width=13cm,height=9cm]{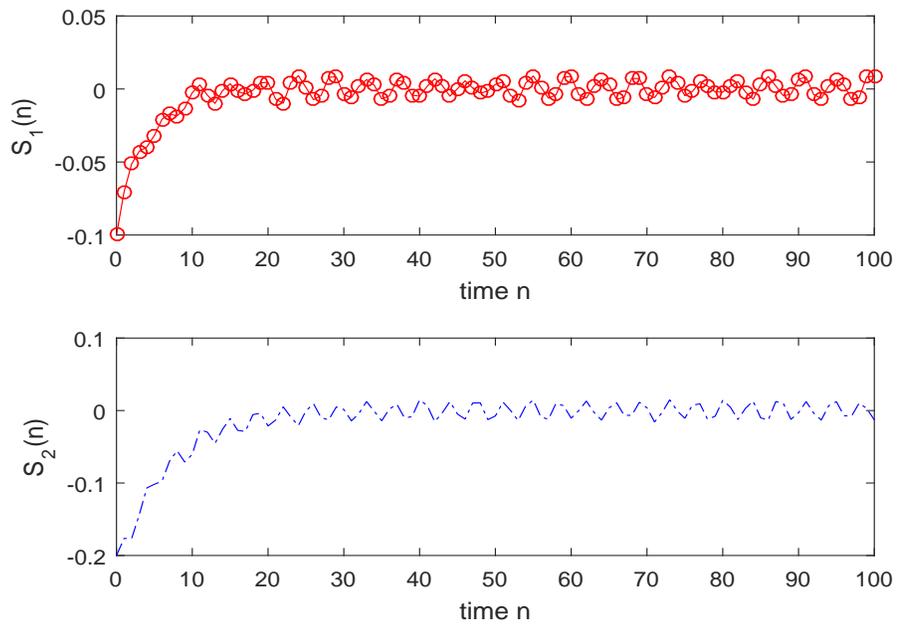}
  \caption{Discrete situation $\mathbb{T}=\mathbb{Z}: S_1,S_2$ with $n$.}
  \end{figure}

  \newpage

\section{ Conclusion}
\indent

In this paper, we propose  a class of neutral type competitive neural networks with mixed time-varying delays and leakage delays on a new type of almost periodic time scales. Based on the exponential dichotomy of linear dynamic equations on time scales, Banach's fixed point theorem and the theory of calculus on time scales, we  obtain  the existence and uniqueness of almost periodic solutions for this class of neural networks without assuming the boundedness of the activation functions and under the same assumptions we also obtain the    global exponential stability of the almost periodic solutions.   Our   approaches of
this paper maybe further be used for other dynamical systems. But, if we modify Definition \ref{def31} to  be the following form:
\begin{definition} \label{def31m}
A time scale $\mathbb{T}$ is called an almost periodic time scale if the set
$$\Lambda_0:=\big\{\tau\in \mathbb{R}:\mathbb{T}_{\tau}\neq \emptyset\big\}\neq\{0\},$$
where $\mathbb{T}_{\tau}=\mathbb{T}\cap\{\mathbb{T}-\tau\}=\mathbb{T}\cap \{t-\tau: t\in \mathbb{T}\}$, and there exists a set $\Lambda$ satisfying
\begin{itemize}
  \item [$(a)$]   $0\in\Lambda\subseteq \Lambda_0$,
  \item [$(b)$] $\Pi(\Lambda)\setminus\{0\}\neq \emptyset$,
  \item [$(c)$] $\widetilde{\mathbb{T}}:=\mathbb{T}(\Pi)=\bigcap\limits_{\tau\in\Pi}\mathbb{T}_\tau\neq \emptyset$,
\end{itemize}where $\Pi:=\Pi(\Lambda)=\{\tau\in \Lambda\subseteq \Lambda_0: \sigma+ \tau\in \Lambda, \forall \sigma\in \Lambda\}$.
\end{definition}
\noindent Then, how to study the almost periodic problem for dynamic equations on the almost periodic time scale defined by  Definition \ref{def31m} with or without condition $(c)$ is a more challenge task.
 This is our future goals.

\end{document}